\documentclass[10pt,oneside]{amsart}

\usepackage{amssymb, amsmath, amsthm, amsfonts}
\usepackage{mathrsfs,comment, mathtools}
\usepackage{stmaryrd}
\input{xypic}
\usepackage{graphicx}
\usepackage{geometry}
\usepackage{float}
\usepackage[all]{xy}

\usepackage[bookmarks,bookmarksopen,bookmarksdepth=4]{hyperref}
\usepackage{url}
\usepackage{enumerate}
\usepackage{color}

\usepackage{verbatim}
\usepackage[T1]{fontenc}
\usepackage{todonotes}

\hypersetup{%
  bookmarksnumbered=true,%
  colorlinks=true,%
  linkcolor=blue,%
  citecolor=blue,%
  filecolor=blue,%
  menucolor=blue,%
  urlcolor=blue,%
  pdfnewwindow=true,%
  pdfstartview=FitBH}

\def\@url#1{{\tt\def~{\lower3.5pt\hbox{\char'176}}\def\_{\char'137}#1}}

\newtheorem{theorem}{Theorem}[section]
\newtheorem{corollary}[theorem]{Corollary}
\newtheorem{proposition}[theorem]{Proposition}
\newtheorem{lemma}[theorem]{Lemma}

\theoremstyle{definition}
\newtheorem{definition}[theorem]{Definition}
\newtheorem{example}[theorem]{Example}
\newtheorem{examples}[theorem]{Examples}
\newtheorem{remark}[theorem]{Remark}

\newtheorem{construction}[theorem]{Construction}

\theoremstyle{plain}

\newtheorem{xxthm}{Theorem}

\newtheorem{xxcor}[xxthm]{Corollary}

\newcommand{\Mdef}[2]{\newcommand{#1}{\relax \ifmmode #2 \else $#2$\fi}}
\Mdef{\bA}{\mathbb{A}}
\Mdef{\bB}{\mathbb{B}}
\Mdef{\bC}{\mathbb{C}}
\Mdef{\bD}{\mathbb{D}}
\Mdef{\bE}{\mathbb{E}}
\Mdef{\bF}{\mathbb{F}}
\Mdef{\bG}{\mathbb{G}}
\Mdef{\bH}{\mathbb{H}}
\Mdef{\bI}{\mathbb{I}}
\Mdef{\bJ}{\mathbb{J}}
\Mdef{\bK}{\mathbb{K}}
\Mdef{\bL}{\mathbb{L}}
\Mdef{\bM}{\mathbb{M}}
\newcommand{\bN}{\mathbb{N}}
\Mdef{\bO}{\mathbb{O}}
\Mdef{\bP}{\mathbb{P}}
\Mdef{\bQ}{\mathbb{Q}}
\Mdef{\bR}{\mathbb{R}}
\Mdef{\bS}{\mathbb{S}}
\Mdef{\bT}{\mathbb{T}}
\Mdef{\bU}{\mathbb{U}}
\Mdef{\bV}{\mathbb{V}}
\Mdef{\bW}{\mathbb{W}}
\Mdef{\bX}{\mathbb{X}}
\Mdef{\bY}{\mathbb{Y}}
\Mdef{\bZ}{\mathbb{Z}}

\Mdef{\bbS}{\mathbb{S}}

\Mdef{\scrA}{\mathscr{A}}
\Mdef{\scrB}{\mathscr{B}}
\Mdef{\scrC}{\mathscr{C}}
\Mdef{\scrD}{\mathscr{D}}
\Mdef{\scrE}{\mathscr{E}}
\Mdef{\scrF}{\mathscr{F}}
\Mdef{\scrG}{\mathscr{G}}
\Mdef{\scrH}{\mathscr{H}}
\Mdef{\scrI}{\mathscr{I}}
\Mdef{\scrJ}{\mathscr{J}}
\Mdef{\scrK}{\mathscr{K}}
\Mdef{\scrL}{\mathscr{L}}
\Mdef{\scrM}{\mathscr{M}}
\Mdef{\scrN}{\mathscr{N}}
\Mdef{\scrO}{\mathscr{O}}
\Mdef{\scrP}{\mathscr{P}}
\Mdef{\scrQ}{\mathscr{Q}}
\Mdef{\scrR}{\mathscr{R}}
\Mdef{\scrS}{\mathscr{S}}
\Mdef{\scrT}{\mathscr{T}}
\Mdef{\scrU}{\mathscr{U}}
\Mdef{\scrV}{\mathscr{V}}
\Mdef{\scrW}{\mathscr{W}}
\Mdef{\scrX}{\mathscr{X}}
\Mdef{\scrY}{\mathscr{Y}}
\Mdef{\scrZ}{\mathscr{Z}}

\Mdef{\mcA}{\mathcal{A}}
\Mdef{\mcB}{\mathcal{B}}
\Mdef{\mcC}{\mathcal{C}}
\Mdef{\mcD}{\mathcal{D}}
\Mdef{\mcE}{\mathcal{E}}
\Mdef{\mcF}{\mathcal{F}}
\Mdef{\mcG}{\mathcal{G}}
\Mdef{\mcH}{\mathcal{H}}
\Mdef{\mcI}{\mathcal{I}}
\Mdef{\mcJ}{\mathcal{J}}
\Mdef{\mcK}{\mathcal{K}}
\Mdef{\mcL}{\mathcal{L}}
\Mdef{\mcM}{\mathcal{M}}
\Mdef{\mcN}{\mathcal{N}}
\Mdef{\mcO}{\mathcal{O}}
\Mdef{\mcP}{\mathcal{P}}
\Mdef{\mcQ}{\mathcal{Q}}
\Mdef{\mcR}{\mathcal{R}}
\Mdef{\mcS}{\mathcal{S}}
\Mdef{\mcT}{\mathcal{T}}
\Mdef{\mcU}{\mathcal{U}}
\Mdef{\mcV}{\mathcal{V}}
\Mdef{\mcW}{\mathcal{W}}
\Mdef{\mcX}{\mathcal{X}}
\Mdef{\mcY}{\mathcal{Y}}
\Mdef{\mcZ}{\mathcal{Z}}

\Mdef{\tA}{\tilde{A}}
\Mdef{\tB}{\tilde{B}}
\Mdef{\tC}{\tilde{C}}
\Mdef{\tE}{\tilde{E}}
\Mdef{\tH}{\tilde{H}}
\Mdef{\tK}{\tilde{K}}
\Mdef{\tL}{\tilde{L}}
\Mdef{\tM}{\tilde{M}}
\Mdef{\tN}{\tilde{N}}
\Mdef{\tP}{\tilde{P}}

\Mdef{\Ab}{\overline{A}}
\Mdef{\Bb}{\overline{B}}
\Mdef{\Cb}{\overline{C}}
\Mdef{\Db}{\overline{D}}
\Mdef{\Eb}{\overline{E}}
\Mdef{\Fb}{\overline{F}}
\Mdef{\Gb}{\overline{G}}
\Mdef{\Hb}{\overline{H}}
\Mdef{\Ib}{\overline{I}}
\Mdef{\Jb}{\overline{J}}
\Mdef{\Kb}{\overline{K}}
\Mdef{\Lb}{\overline{L}}
\Mdef{\Mb}{\overline{M}}
\Mdef{\Nb}{\overline{N}}
\Mdef{\Ob}{\overline{O}}
\Mdef{\Pb}{\overline{P}}
\Mdef{\Qb}{\overline{Q}}
\Mdef{\Rb}{\overline{R}}
\Mdef{\Sb}{\overline{S}}
\Mdef{\Tb}{\overline{T}}
\Mdef{\Ub}{\overline{U}}
\Mdef{\Vb}{\overline{V}}
\Mdef{\Wb}{\overline{W}}
\Mdef{\Xb}{\overline{X}}
\Mdef{\Yb}{\overline{Y}}
\Mdef{\Zb}{\overline{Z}}

\newcommand{\co}{\colon}

\def\endash{\mathchar"2D}

\newcommand{\leftmod}{\endash \textnormal{mod}}

\newcommand{\lra}{\longrightarrow}

\Mdef{\id}{\mathrm{Id}}

\newcommand{\adjunction}[4]{
\xymatrix{
#1:#2 \ar@<0.7ex>[r] &
\ar@<0.7ex>[l] #3:#4
}}

\Mdef{\bhom}{\mathbf{\hat{H}om}}

\Mdef{\Mod}{\mathrm{mod}}

\newcommand{\ilim}{\mathop{ \mathrm{lim }} }
\newcommand{\colim}{\mathop{  \mathrm{colim }} }

\DeclareMathOperator{\mackeyfunctor}{Mackey}
\DeclareMathOperator{\sheaffunctor}{Sheaf}

\newcommand{\weylsheaf}[1]{\textnormal{Weyl} \endash #1 \endash \textnormal{sheaf}_{\bQ}(\sub #1)}

\newcommand{\Reqsheaf}[3]{#1 \endash \textnormal{Sheaf}_{#3}(#2)}

\newcommand{\Reqpresheaf}[3]{#1 \endash \textnormal{PreSheaf}_{#2}(#3)}
\newcommand{\Rweylsheaf}[2]{\textnormal{Weyl} \endash #1 \endash \textnormal{sheaf}_{#2}(\sub #1)}

\newcommand{\seteqsheaf}[2]{#1 \endash \textnormal{Sheaf}(#2)}
\newcommand{\seteqpresheaf}[2]{#1 \endash \textnormal{PreSheaf}(#2)}

\newcommand{\setsheaf}[1]{\textnormal{Sheaf}(#1)}

\newcommand{\sheafify}{\ell}

\newcommand{\adic}[1]{\bZ_{#1}^{\wedge}}

\newcommand{\mackey}[1]{\textnormal{Mackey}_{\bQ}(#1)}

\newcommand{\sub}{\mathcal{S}}

\newcommand{\sets}{\textnormal{Sets}}

\newcommand{\Gmoddisc}[2]{#1_d[#2] \leftmod}
\newcommand{\Gmod}[2]{#1[#2] \leftmod}

\newcommand{\disc}{\textnormal{disc}}

\newcommand{\inc}{\textnormal{inc}}

\newcommand{\opensub}{\underset{\textnormal{open}}{\leqslant}}

\newcommand{\opennormalsub}{\underset{\textnormal{open}}{\trianglelefteqslant}}

\DeclareMathOperator{\stab}{\textnormal{stab}}

\DeclareMathOperator{\coker}{\textnormal{Coker}}

\DeclareMathOperator{\const}{\textnormal{Const}}

\DeclareMathOperator{\core}{\textnormal{Core}}
\DeclareMathOperator{\Weyl}{\textnormal{Weyl}}

\title{Equivariant sheaves for profinite groups}
\author[Barnes]{David Barnes}
\address[Barnes]{Mathematical Sciences Research Centre, Queen's University Belfast}
\author[Sugrue]{Danny Sugrue}
\address[Sugrue]{Mathematical Sciences Research Centre, Queen's University Belfast}

\thanks{The authors extend their thanks to Scott Balchin and Tobias Barthel for 
many useful conversations and a careful reading of a draft. 
The second author gratefully acknowledges support from the
Engineering and Physical Sciences Research Council under Grant 1631308.
}

\begin{document}
\begin{abstract}
We study equivariant sheaves over profinite spaces, where the group is also taken to be profinite.  
We resolve a serious deficit in the existing theory by constructing a good notion of 
equivariant presheaves, with a suitable equivariant sheafification functor. 
Using equivariant sheafification, we develop the general theory of equivariant sheaves of modules over a ring, 
give explicit constructions of infinite products and introduce an equivariant analogue of skyscraper sheaves. 

These results underlie recent work by the authors 
which proves that there is an algebraic model for rational $G$-spectra in terms of 
equivariant sheaves over profinite spaces. 
That model is constructed in terms of 
Weyl-$G$-sheaves over the space of closed subgroups of $G$, 
where the term Weyl indicates that the stalk over $H$
is $H$-fixed.   
In this paper, we prove that Weyl-$G$-sheaves of $R$-modules 
form an abelian category with enough injectives
and is a coreflective subcategory of equivariant sheaves of $R$-modules. 

We end the paper with a structural result that provides another way to conveniently build
equivariant sheaves from simpler data. 
We prove that a $G$-equivariant sheaf over a profinite base space $X$
is a colimit of equivariant sheaves over 
finite discrete spaces $X_i$ with actions of finite groups $G_i$, where
$X$ is the limit of the $X_i$ and $G$ is the limit of the $G_i$. 
\end{abstract}

\maketitle

\setcounter{tocdepth}{1}
\tableofcontents

\section{Introduction}
Sheaves of modules over a ring combine algebra and topology
in a manner that has been highly useful to algebraic topology and 
algebraic geometry.  
Adding the action of a topological group $G$ is a natural step
and creates a richer structure. For example, 
each stalk has an action of the corresponding stabiliser group
(similar to equivariant vector bundles).

The first difficulty in working with equivariant sheaves is the lack of a theory of 
equivariant presheaves for infinite topological groups.
Both definitions of an equivariant sheaf 
(Definitions \ref{defn:cocyclesheaf} and \ref{defn:eqsheaf})
use the sheaf space to ensure that the group action is continuous. 
Since a presheaf does not have a topological structure, there is no
way to incorporate a non-trivial topology on the group.
This is a serious deficit. Without a good theory or presheaves and sheafification
in the equivariant setting, it is hard to replicate 
many of the standard constructions of sheaf theory
(such as infinite coproducts or push forwards) 
and it is more difficult to construct examples.

In this paper we focus on the case where the group and base space are profinite.
We use the interaction between these two conditions to 
give a definition of equivariant presheaves and equivariant sheafification.
We build on this to develop the theory of equivariant sheaves,
construct adjunctions with other categories 
(each such adjunction gives a class of examples of equivariant sheaves) 
and give explicit constructions of infinite products.

Though equivariant sheaves have been studied before, that literature 
takes a very different direction to our work. For example, in the book of 
Bernstein and Lunts \cite{bernlunts}, 
the authors want to make
a derived category of sheaves over a free replacement of the base space.
This is far too restrictive for our purposes. Indeed, 
$\sub G$  --- the space of closed subgroups of $G$, under conjugation by $G$ --- 
is the base space of greatest interest for our applications and 
will always have fixed points.

A summary of the major results is given in Subsection \ref{subsec:mainresults}.
The original motivation for this work comes from the question of finding algebraic models
for rational $G$-spectra, where $G$ is a profinite group, see 
Subsection \ref{subsec:motivation} for further details. 
This paper is a combination of results from the PhD thesis of the second author 
(which was produced under supervision of the first) and substantial further additions.

\subsection{Summary of main results}\label{subsec:mainresults}
To produce a good notion of $G$-equivariant presheaves, we replace the
topological condition that $G$ acts continuously on the sheaf space with 
a property that can be expressed (and hence verified) algebraically. 
The starting point is Proposition \ref{prop:compactdiscrete}.
This proposition states that for a profinite group $G$, 
a section of a $G$-equivariant sheaf over a profinite base space $X$ can be restricted to 
an $N$-equivariant section, for some open subgroup $N$ of $G$. 
We describe this as saying that ``every section is locally sub-equivariant''
and have the following useful rephrasing. 

\begin{xxcor}[Corollary \ref{cor:sectionquot}]
Let $F$ be a $G$-equivariant sheaf of sets over a profinite $G$-space $X$.
For $U$ a compact open subset of $X$, the set of sections
$F(U)$ is a discrete $\stab_G(U)$-set.
\end{xxcor}

A related result can be found in work of Schneider, see \cite[Statement (3)]{schneider98}.
The discrete action encodes the continuity of the stabiliser action 
on sections. As the following theorem shows, it is sufficient to 
ensure that associated sheaf space has a continuous action of a profinite group $G$. 

As $X$ is profinite, the set $\mcB$ of compact-open sets of $X$ is a basis for 
the topology on $X$. Following Definitions \ref{def:probasisdiagram} 
and \ref {def:gpresheaf},
let $\seteqpresheaf{G}{\mcB}$ be the category of $G$-equivariant presheaves over $\mcB$.

\begin{xxthm}[Theorems \ref{thm:equivariantsheafspace} and \ref{thm:equisheafifyfunctor}]
There is a notion of an equivariant presheaf and a 
construction of equivariant sheafification which is left adjoint to the forgetful functor 
\[
\adjunction{\sheafify}{\seteqpresheaf{G}{\mcB}}{\seteqsheaf{G}{X}}{\textnormal{forget}.}
\]
The functor $\textnormal{forget} \circ \sheafify$ is idempotent, hence the forgetful functor is fully faithful.
\end{xxthm}

For $R$ a ring, these results extend to $G$-sheaves of (left) $R$-modules and can be used to 
show that we have an abelian category and give explicit constructions of 
(small) colimits and limits, including infinite products.

\begin{xxthm}[Theorem \ref{Gsheafabelian} and Proposition \ref{prop:infproduct}]
The category of $G$-equivariant sheaves of $R$-modules over a profinite $G$-space $X$ is an abelian category
with all small limits and colimits. 
\end{xxthm}

In Section \ref{sec:pbpf} we use our results on equivariant presheaves 
to construct the change of base space adjunction (pull backs and push forwards).
The push forward functor allows us to construct an equivariant 
version of a skyscraper sheaf in Example \ref{Gskyscraper}. 
Such a sheaf has non-zero stalks over exactly one orbit of the base space $X$ 
(recall that in equivariant contexts, orbits take the place of points). 
This construction is adjoint to the functor which takes a stalk at a point of $X$.
The class of equivariant skyscrapers are essentially the only example one needs, 
as the following makes clear.

\begin{xxthm}[Theorem \ref{thm:enoughinj}]
The injective equivariant skyscraper sheaves provide enough injectives for 
the category of $G$-equivariant sheaves of $R$-modules over $X$.
\end{xxthm}

We now apply these results to study Weyl-$G$-sheaves:
the category of $G$-equivariant 
sheaves of $R$-modules over $\sub G$ (the space of closed subgroups of $G$) 
such that the stalk over $K$ is $K$-fixed. 
This category was central to \cite{BSmackey}. 
We show that we can construct a Weyl-$G$-sheaf from a $G$-equivariant sheaf over $\sub G$
by taking stalk-wise fixed points.

\begin{xxthm}[Lemma \ref{lem:weyladjunction} and Theorem \ref{thm:weylabelian}]
There is an adjunction
\[
\adjunction{\inc}{\Rweylsheaf{G}{R}}{\Reqsheaf{G}{\sub G}{R}}{\Weyl.}
\]
The functor $\inc \circ \Weyl$ is idempotent, hence the inclusion functor is fully faithful.

The category of Weyl-$G$-sheaves of $R$-modules is an abelian category with all
small colimits and limits. 
Limits are constructed by taking the limit of the underlying diagram 
in $\Reqsheaf{G}{\sub G}{R}$ and applying the functor $\Weyl$.
\end{xxthm}

In Section \ref{sec:diagramsheaves} we produce another method to construct 
$G$-equivariant sheaves over a profinite space $X$. The idea is to combine 
a compatible sequence of equivariant sheaves over finite discrete spaces.
See Theorem \ref{thm:diagramsheafcolim} for the general case and 
Corollary \ref{cor:weyldiagrams} for this result applied to the case of 
Weyl-$G$-sheaves. 
This method, alongside equivariant sheafification, 
shows how we can conveniently build equivariant sheaves
from simple algebraic information,
see Examples \ref{ex:groupring}, \ref{ex:groupring2}
and \ref{ex:fixedpoints}.

We end the paper with two Appendices recapping the definitions and basic results of
profinite groups, profinite spaces and sheaves. 

\subsection{Motivation}\label{subsec:motivation}

Profinite groups are a commonly encountered class of compact topological groups, 
appearing most often as Galois groups or when one has a diagram of finite groups. 
As a first example, the Morava stabilizer group $\mathbb{S}_n$ 
from chromatic homotopy theory is profinite.
Secondly, number theory makes substantial use of profinite groups, as seen in 
Bley and Boltje \cite{BB04}.
Thirdly, the \'{e}tale fundamental groups of algebraic geometry are profinite. 
This ubiquity drives our interest in rational $G$-spectra for profinite $G$.


Sheaves over profinite spaces occur in questions of algebraic models 
beyond the case of profinite groups.
Indeed, Greenlees' conjectured \cite{Gconjecture}
that for any compact Lie group $G$
there exists an algebraic model for rational $G$-spectra
built in terms of sheaves over  $\sub G$, the space of subgroups of $G$.  
The quotient space under the conjugation action, $\sub G/G$ 
is a profinite space by Lewis et al. \cite[Lemma V.2.8]{lms86}.
It follows that a good theory of $G$-sheaves over $\sub G$
and the related category of sheaves over $\sub G/G$
for $G$ a compact Lie group will be needed.
By developing the theory for profinite $G$, 
the current paper provides a guide for the 
compact Lie case. 

In \cite{BSmackey}, the authors construct an equivalence between the category of 
rational $G$-Mackey functors and Weyl-$G$-sheaves.
This required the definition of equivariant presheaves and skyscrapers.
In \cite{sugruethesis} and \cite{BSspectra}, this 
equivalence is used to prove the main theorem of this line of work:
a classification of rational $G$-spectra in terms of rational Weyl-$G$-sheaves. 
Such a classification is known as an algebraic model for rational $G$-spectra.
For an introduction to algebraic models see \cite{BKclassify} 
by the first author and K\k{e}dziorek.
The key points to take away are that an algebraic model allows one to construct 
objects and calculate maps in (the homotopy category of) rational $G$-spectra 
using the tools from (homological) algebra 
while avoiding the complexity of working in a topological setting.



\section{Equivariant sheaves}\label{sec:equivsheaves}

In this section we 
introduce the category of $G$-equivariant sheaves of sets. 
See Appendix \ref{Sec:sheavespresheaves} for a recap of sheaves and choice of terminology.
There are several equivalent definitions, we start by giving a version from 
Bernstein and Lunts \cite{bernlunts} in terms of pullback sheaves, see Definition \ref{defn:pullpushfunctors}. 
See also Scheiderer \cite[Chapter 8]{Scheiderer94}.

Let $G$ be a topological group and $X$ a topological space with a left $G$-action. 
We name several maps relating to this action
\[
\xymatrix{
G \times G \times X 
\ar@<+2pt>@/^10pt/[r]|-{e_0}
\ar[r]|-{e_1}
\ar@<-2pt>@/_10pt/[r]|-{e_2}
& 
G \times X 
\ar@<+2pt>@/^10pt/[r]|-{d_0}
\ar@<-2pt>@/_10pt/[r]|-{d_1}
& 
X
\ar[l]|-{s_0}
}
\]
\begin{align*}
e_0(g, h, x) &= (h, g^{-1} x) 
&
d_0(g,x) &= g^{-1} x \\
e_1(g, h, x) &= (gh, x)
&
d_1(g,x) &=  x \\
e_2(g, h, x) &= (g, x)
&
s_0(x) &= (e,x)
\end{align*}

\begin{definition}\label{defn:cocyclesheaf}
A $G$-\emph{equivariant sheaf} of sets over $X$ is a sheaf of sets $F$ over $X$ together with an isomorphism 
\[
\rho \co d_1^*(F)\lra d_0^*(F)
\]
such that the cocycle condition and the identity condition hold:
\[
e_1^* \rho = e_0^* \rho \circ e_2^* \rho
\quad \quad
s_0^* \rho = \id_F.
\]
A morphism $f \co F \to F'$ of $G$-equivariant sheaves of sets over $X$ is a map $f$ of sheaves 
of sets such that the following square commutes. 
\[
\xymatrix{
d_1^*(F)
\ar[d]^{\rho}\ar[r]^{d_1^*(f)} &
d_1^*(F')
\ar[d]^{\rho'}
\\ 
d_0^*(F)
\ar[r]^{d_0^*(f)}  &  
d_0^*(F')}
\]
We write $\seteqsheaf{G}{X}$ for this category and often shorten 
the name $G$-equivariant sheaf to $G$-\emph{sheaf}.
\end{definition}

For each pair $(g,x) \in G \times X$, we can calculate the values on the stalks (Definition \ref{def:sheafandstalk})
\[
d_1^*(F)_{(g,x)} =\left\lbrace (g,x,f)\mid f\in F_x \right\rbrace\cong F_x 
\quad \text{and} \quad 
d_0^*(F)_{(g,x)} =\left\lbrace (g,x,f)\mid f\in F_{g^{-1} x} \right\rbrace  \cong F_{g^{-1} x}.
\]
It follows that $\rho$ induces isomorphisms on stalks 
\begin{align*}
\rho_{(g,x)} \colon  F_x \cong d_1^*(F)_{(g,x)} \longrightarrow d_0^*(F)_{(g,x)} \cong F_{g^{-1}x}.
\end{align*} 
Taking the inverse gives maps $g \colon F_x \lra F_{gx}$ for any $g$ and $x$. 
The cocycle and identity conditions imply that these maps are compatible under composition, 
so that 
\[
F_x \overset{g}{\lra} F_{gx} \overset{g'}{\lra} F_{g'gx}
\]
is the same as the map $g'g$ and $e$ acts as the identity. 
Consequently, $F_x$ has an action of $\stab_G (x)$.

There is a second definition of $G$-equivariant sheaves, which is more focused 
on the local homeomorphism definition of a sheaf.

\begin{definition}\label{defn:eqsheaf}
A \emph{$G$-equivariant sheaf} of sets over a topological space $X$ is a 
$G$-equivariant local homeomorphism $p \co E \lra X$ 
of topological spaces with continuous $G$-actions.

We will write this as either the pair $(E,p)$ or simply as $E$.
A map $f \co (E,p) \lra (E',p')$ of $G$-equivariant sheaves of sets over $X$ is a 
$G$-equivariant map $f \co E \to E'$ such that $p'f=p$.  
\end{definition}

These two definitions of sheaves give equivalent categories. 
Just as one eventually ceases to distinguish between 
sheaves and sheaf spaces, one may freely use either definition
of an equivariant sheaf.

\begin{theorem}\label{thm:eqsheafequiv}
The category of equivariant sheaves of Definition \ref{defn:cocyclesheaf}
is equivalent to the category of equivariant sheaves of Definition \ref{defn:eqsheaf}.
\end{theorem}
\begin{proof}
A proof in the case of sheaves of $\bQ$-modules can be found in 
work of the second author \cite[Theorem 4.1.5]{sugruethesis} and 
K{\k{e}}dziorek \cite[Lemma 7.0.3]{kedziorekthesis}. 
The core of the argument holds in the case of sheaves of sets.
\end{proof}

As non-equivariantly, we have a notion of sections of a sheaf. 

\begin{definition}
Let $p \co E \to X$ be a a $G$-sheaf and $U$ an open subset of $X$. 
Define $\Gamma(U,E)$ to be the set of (non-equivariant) sections $s \co U \lra E$.

Given $s \in \Gamma(U,E)$ and $g \in G$, define a section of $gU$
\[
g \ast s =  g \circ s \circ g^{-1} \co g U \lra E
\]
which sends $v=gu$ to $gs(u) = gs(g^{-1} v)$. 

Let $H$ be a subgroup of $G$ such that $U$ is invariant under the action of a subgroup $H$
(that is, $h U = U$ for all $h \in H$). 
We say that $s \in \Gamma(U,E)$ is \emph{$H$-equivariant} 
if $h \ast s=s$ for all $h \in H$. 
\end{definition}

Take a section $s$ over $U$ of a $G$-sheaf $p \co E \to X$, then for $x \in X$, $g \in G$
\[
g(s_x) = (g \ast s)_{gx}
\]
which is the key formula for relating the action of $G$ on sections and stalks.

Confirmation that we are correct in taking the sections to be non-equivariant is given by the following,
which follows from the proof of Theorem \ref{thm:eqsheafequiv}.

\begin{corollary}\label{cor:sectionsaresections}
Given a $G$-sheaf $F$ as in Definition \ref{defn:cocyclesheaf} we have a 
corresponding $G$-sheaf space $E$ from Definition \ref{defn:eqsheaf}. 
For $U$ an open subspace of $X$, there is a canonical bijection
\[
F(U) \cong \Gamma(U,E)
\]
and this bijection is natural in inclusions of subsets $U \to V$
and isomorphisms $g \co U \to gU$.
\end{corollary}

\begin{lemma}\label{lem:discretestalks}
Given a $G$-sheaf $(E,p)$ over $X$ and 
$x \in X$, the stalk $E_x$ is a discrete $\stab_G(x)$-set (in the sense of Definition \ref{def:modulediscrete}).
\end{lemma}
\begin{proof}
Given a $G$-sheaf $(E,p)$ over $X$ and 
$x \in X$, the stalk $E_x$ is a discrete subspace of $E$,
with a continuous action of $\stab_G(x)$.
\end{proof}

\begin{remark}
Let $G_\delta$ be the group $G$, but with the discrete topology.
Then a $G$-space $X$ can be considered as a $G_\delta$-space via the 
identity map $G_\delta \to G$. 
An interesting question is how $G_\delta$-sheaves over $X$ and 
$G$-sheaves over $X$ are related. 
Scheiderer provides an answer in \cite[Remark 8.3]{Scheiderer94},
which states that a $G_\delta$-sheaf over a $G$-space $X$ is a $G$-sheaf if and only if
$g s_x = s_{gx}$ for all $g$ in some neighbourhood of the identity of $G$.
\end{remark}

\begin{example}\label{ex:onepointsheafsets}
If we let $X$ be the trivial one point $G$-space, 
then a $G$-sheaf of sets over $X$ is a discrete $G$-set.

Generalising, let $X$ be a discrete $G$-space, the category of 
$G$-sheaf of sets over $X$ is equivalent to the product category
\[
\prod_{[x] \in X/G } \textnormal{discrete} \endash \stab_G(x) \endash \textnormal{sets}.
\]
\end{example}

\begin{definition}
A $G$-sheaf of sets $E$ over a $G$-space $X$ is said to be \emph{stalk-wise fixed} 
if the action of $\stab_G(x)$ on $E_x$ is trivial for each $x\in X$. 
\end{definition}

The evident example of a stalk-wise fixed sheaf is a constant sheaf. This example also demonstrates
that the $G$-action on the sheaf space of a stalk-wise fixed sheaf is not trivial in general.

\begin{example}\label{ex:constantsheafset}
Let $X$ be a $G$-space and $A$ a set (equipped with the discrete topology). 
The \emph{equivariant constant sheaf} at $A$ over $X$ is defined by the projection onto the $X$ factor, 
which is a local homeomorphism
\[
p \co A \times X \lra X.
\]
\end{example}

\section{Sections of equivariant sheaves}\label{sec:sections}

From this section onwards we will assume that both $G$ and $X$ are profinite.
See Appendix \ref{Sec:profinitestuff} for a recap of the profinite spaces and groups. 
This assumption allows us to show that any section of a $G$-equivariant is 
\emph{locally sub-equivariant} in the sense of the following result. 
A similar result can be found in Schneider \cite[Statement (3)]{schneider98}.

\begin{proposition}\label{prop:compactdiscrete}
Let $F$ be a $G$-sheaf of sets over a profinite $G$-space $X$
and let $U$ be an open subset of $X$. 
Any section of $F(U)$ restricts to an $N$-equivariant section 
over an $N$-invariant domain, for some open normal subgroup $N$ of $G$.
\end{proposition}
\begin{proof}
Write $p \co E \lra X$ be the sheaf space description of $F$.
Given a section $s \colon U\lra E$, we can find a compact-open subset of 
the form $V$ contained in $U$ since $X$ has an open-closed basis. 
By Lemma \ref{lem:restrictequi}, $V$ is $N_1$-invariant for some open normal subgroup $N_1$ of $G$.

The set $s(V) \subset E$ is compact and open, hence  
another application of Lemma \ref{lem:restrictequi} gives an 
open normal subgroup $N_2$ such that $s(V)$ is $N_2$-invariant.

Set $N= N_1 \cap N_2$, we claim that $s$ is $N$-equivariant. 
As $V$ and $s(V)$ are invariant under $N$, 
we have $ny \in V$ and $ns(y) \in s(V)$ for all $y \in V$ and $n \in N$. 
Using this and the fact that $p$ is a $G$-equivariant injection on $V$, we have 
\[
p \circ s(ny)=ny \quad  \quad p(ns(y))=np(s(y))=ny,
\]
and so $s(ny) = ns(y)$ for all $y \in V$ and $n \in N$.
\end{proof}

The next result is key to Definition \ref{def:gpresheaf}, which defines 
equivariant presheaves.
That definition encodes the continuity of the $G$ action on a sheaf space in
terms of its actions on sections, as Theorem  \ref{thm:equisheafifyfunctor} implies. 

\begin{corollary}\label{cor:sectionquot}
Let $F$ be a $G$-sheaf of sets over a profinite $G$-space $X$.
For $U$ a compact open subset of $X$, the set of sections
$F(U)$ is a discrete $\stab_G(U)$-set.
\end{corollary}
\begin{proof}
Let $E$ be the sheaf space corresponding to $F$.
Let $s\in \Gamma\left(U,E\right)=F(U)$, by Proposition \ref{prop:compactdiscrete} there exists an open normal 
subgroup $N$ such that $U$ is $N$-invariant and $s$ is $N$-equivariant. 
That is, $s$ is fixed by the action of $N$ and $N \trianglelefteqslant \stab_G(U)$. 
As $N$ is open, we see that $F(U) = \Gamma\left(U,E\right)$ is discrete.
\end{proof}

\section{Equivariant presheaves}\label{sec:equivpresheaves}

As we saw, it was uncomplicated to define a $G$-action on a sheaf space: one simply requires the group
to act on compatibility on the base and total space in a continuous manner.
Definition \ref{defn:cocyclesheaf} defines an equivariant sheaf in terms
of an isomorphism between certain pullback sheaves and the cocycle condition, 
which implicitly uses the sheaf space of sheaf.
It is harder to directly define a $G$-equivariant presheaf, as  
the non-equivariant version has no topology on which the group can act continuously. 
This is a serious deficit. Without a good theory or presheaves and sheafification it is hard to 
replicate many of the standard constructions of sheaf theory, such as 
the infinite coproducts of Section \ref{sec:colimitssheaves} 
or the push forwards of Section \ref{sec:pbpf}.

For finite groups, there is a long-known description of a equivariant presheaves
\cite[Example A2.1.11(c)]{elephant}. Given a finite group $G$ acting on the 
topological space $X$, define the category $\mcO_G(X)$ to have objects the 
open sets of $X$ with morphisms
\[
\mcO_G(X)(U,V) = \{ g \in G  \mid gU \subseteq V \}.
\]
A $G$-equivariant presheaf is a functor 
\[
F \co \mcO_G(X)^{op} \lra \sets.
\]
We can think of this as a presheaf with additional maps 
$F(U) \to F(gU)$ for each $g \in G$. 
For a general topological group $G$ and topological space
this definition is insufficient, as it will not distinguish between 
$G$ with its given topology and $G$ with the discrete topology. 
We resolve this issue when $G$ and $X$ are profinite using Corollary \ref{cor:sectionquot}.
We construct a good definition of equivariant presheaf and 
an equivariant sheafification functor with properties similar to the
non-equivariant case. 

A different approach appears in \cite{schneider98}. 
The idea is to start with $G_{d}$-sheaves for $G_d$ the group $G$ 
with the discrete topology and define a subsheaf $F'$ of a $G_d$-sheaf
$F$ consisting of those sections which are locally sub-equivariant. 

As before, let $G$ be a profinite group and $X$ a profinite $G$-space. Note that since 
$X$ is profinite, the set of all open-compact subsets forms a basis for the topology. 
\begin{definition}\label{def:probasisdiagram}
For $X$ a profinite $G$-space with open-closed basis
$\mcB$, the category $\mcO_G(\mcB)$ has objects given by 
the elements of $\mcB$ and morphisms given by the set
\[
\mcO_G(\mcB)(U,V) = \{ g \in G  \mid gU \subseteq V \}.
\]
\end{definition}

Any map in $\mcO_G(\mcB)$ is a composite of an \emph{inclusion} map
$e \co U \to V$ and a \emph{translation} map $g \co U \to gU$. Moreover, these maps commute in the sense that 
\[
\xymatrix{
U \ar[r]^e  \ar[d]_g &
V \ar[d]^g \\
gU \ar[r]_e &
gV.
}
\]
The horizontal maps will induce the restriction maps 
that occur in non-equivariant sheaves. 

We now define $G$-equivariant presheaves in terms of an open-closed basis, the choice of which is 
(largely) unimportant by Remark \ref{rmk:basischoice}.

\begin{definition}\label{def:gpresheaf}
For $\mcB$ a open-closed basis of $X$, a 
\emph{$G$-equivariant presheaf of sets over $\mcB$} is a functor 
\[
F \co \mcO_G(\mcB)^{op} \lra \sets
\]
such that the action of $\mcO_G(\mcB)^{op}(U,U)=\stab_G(U)$ on $F(U)$ is 
continuous, when $F(U)$ is given the discrete topology and 
$\stab_G(U) \leqslant G$ has the subspace topology.

A map of $G$-presheaves of sets over $\mcB$ is a natural transformation.
We write $\seteqpresheaf{G}{\mcB}$ to denote 
the category of $G$-presheaves of sets over $\mcB$.
\end{definition}

Expanding the definition, $F(U)$ is a a discrete $\stab_G(U)$-set under the action 
\[
F(g) \co F(U) \lra F(U)
\]
for $g \in \stab_G(U)$.  For the sake of notational simplicity, we often write
\[
g \coloneqq F(g^{-1}) \co F(U) \lra F(gU)
\]
so that $hg$ is a map from $F(U)$ to $F(hg U)$.

Let $U \subseteq V$ be open sets and $g \in G$. Then 
for a $G$-presheaf $F$, 
the following diagram commutes
\[
\xymatrix{
F(V) \ar[r]^{F(e)}  \ar[d]_{g=F(g^{-1})} &
F(U) \ar[d]^{g=F(g^{-1})} \\
F(gV) \ar[r]_{F(e)} &
F(gU).
}
\]
We can think of this diagram as saying that the restriction maps are 
equivariant.

\begin{theorem}\label{thm:equivariantsheafspace}
Let $\mcB$ be an open-closed basis of the profinite $G$-space $X$ and let $F$ be a $G$-presheaf of sets over $\mcB$.
Then the sheaf space constructed from $F$ considered as a non-equivariant 
presheaf has a canonical $G$-action. 
\end{theorem}
\begin{proof}
We define a topological space $E$ in the same way as the usual sheafification construction. The underlying set is 
\[
E=\coprod_{x \in X} F_x
\]
with projection map $p \co E \to X$ sending $f \in F_x$ to $x$. 
The topology is generated by the sets 
\[
B(s,U) = \{ s_x \mid x \in U\}
\]
for $U \in \mcB$ and $s \in F(U)$. 
The map $p$ is a local homeomorphism.

To define the $G$-action, take a germ $f_x \in F_x$ and $g \in G$.
Choose a section $s \in F(U)$ representing $f_x$, let 
$g \ast s \in F(gU)$ denote the image of $s$ under $g=F(g^{-1})$.
We define
\[
g (f_x) = (g \ast s)_{gx}
\] 
This action is well-defined as restriction is equivariant and
any two sections which represent a germ
agree on some neighbourhood. 
The action defines a map $m \co G \times E \lra E$
such that the projection map is $G$-equivariant. 

It remains to show that the $G$-action on $E$ is continuous. 
Let $B(s,U)$ be a basis element, and $(g,e_y) \in G \times E$ such that
\[
m(g,e_y) =s_x \in B(s,U).
\]
Then we see that $gy=x$, so $y=g^{-1} x$. 
Let $t \in F(W)$ be a representative of $e_y$, 
then 
\[
s_x = g e_y = (g \ast t)_{gy} = (g \ast t)_x 
\]
so there is some compact-open set $V \subseteq U \cap gW$
where $s_{|V} = (g\ast t)_{|V} = g \ast (t_{|g^{-1} V})$ in $F(V)$.
Let $s' = s_{|V}$ and $t' =t_{|g^{-1} V}$, then $s'=g\ast t'$. 

We know that $F(V)$ is a discrete $\stab_G(V)$-set.
Hence, there is an open normal subgroup $N$ of $\stab_G(V)$ 
such that for any $n \in N$, $n\ast s'=s'=g\ast t' = ng\ast t'$. 
By Lemma \ref{lem:openstab}, the group $\stab_G(V)$ is open, 
hence $N$ is an open subgroup of $G$ and $Ng$ is an open neighbourhood of $g$ in $G$.

We claim that the open set $Ng \times B(t', g^{-1} V) \subseteq G \times E$ is sent to 
$B(s,U)$ by $m$. We see that
\[
m(ng, t'_{g^{-1} v}) = (ng\ast t')_{nv} = (gt')_{nv}  = s'_{nv}=s_{nv}
\]
and $nv \in NV = V \subseteq U$, so the claim is true. 
\end{proof}

\begin{definition}
Let $X$ be a profinite $G$-space with open-closed basis
$\mcB$ and $F$ a $G$-presheaf of sets over $\mcB$.
The $G$-sheaf constructed in Theorem \ref{thm:equivariantsheafspace}
is called the \emph{equivariant sheafification} of $F$ and denoted $\sheafify F$.
\end{definition}

\begin{theorem}\label{thm:equisheafifyfunctor}
Equivariant sheafification defines a functor 
\[
\sheafify \co \seteqpresheaf{G}{\mcB} \lra \seteqsheaf{G}{X}
\]
which is left adjoint to the forgetful functor.

The unit of this adjunction gives a 
canonical map from a $G$-presheaf to a $G$-sheaf $F \to \sheafify F$. 
The canonical map $\sheafify F \to \sheafify^2 F$ is an isomorphism. 
Furthermore, if $E$ is a $G$-sheaf, a map $F \to E$ factors over $F \to \sheafify F$.
\end{theorem}
\begin{proof}
Corollary \ref{cor:sectionquot} implies that the forgetful functor exists.
The other statements follow from the non-equivariant case. 
\end{proof}

\begin{remark}\label{rmk:basischoice}
If we are interested in $G$-sheaves, then we claim that the choice of the basis $\mcB$ is unimportant. 
A presheaf over an open-closed basis $\mcB$ defines a presheaf on an open-closed basis 
$\mcB' \supseteq \mcB$ by right Kan extension, as in Lemma \ref{lem:sheafoverbasis}.
The (equivariant) sheafifications of those two presheaves agree.
Since any two open-closed bases $\mcB$ and $\mcB'$ can be included into a larger open-closed basis,
the claim is justified. 
\end{remark}

\section{\texorpdfstring{$G$}{G}-sheaves of \texorpdfstring{$R$}{R}-modules}\label{sec:GsheafRmod}

From this section onwards, we will let $R$ denote a fixed ring (with unit).
We continue to assume that $G$ is a profinite group and $X$ is a profinite $G$-space.
We introduce the category of $G$-sheaves of $R$-modules over $X$. 
As earlier, we have two equivalent definitions, once these are introduced and 
related to $G$-equivariant presheaves of $R$-modules, we will  
verify that the category of $G$-sheaves of $R$-modules is an abelian category.

\begin{definition}\label{defn:rmodulecocyclesheaf}
A $G$-\emph{equivariant sheaf} of $R$-modules over $X$ is sheaf of $R$-modules $F$ over $X$ together with an isomorphism 
\[
\rho \co d_1^*(F)\lra d_0^*(F)
\]
such that the cocycle condition and the identity condition hold:
\[
e_1^* \rho = e_0^* \rho \circ e_2^* \rho
\quad \quad
s_0^* \rho = \id_F.
\]
A morphism $f \co F \to F'$ of $G$-sheaves of $R$-modules over $X$ is a map $f$ of sheaves 
of $R$-modules such that the following square commutes. 
\[
\xymatrix{
d_1^*(F)
\ar[d]^{\rho}\ar[r]^{d_1^*(f)} &
d_1^*(F')
\ar[d]^{\rho'}
\\ 
d_0^*(F)
\ar[r]^{d_0^*(f)}  &  
d_0^*(F')}
\]
We write $\Reqsheaf{G}{X}{R}$ for this category. 
\end{definition}

\begin{definition}\label{defn:eqsheafRmod}
A \emph{$G$-equivariant sheaf} of $R$-modules over a topological space $X$ is a map
$p \co E \to X$ such that:
\begin{enumerate}
\item \label{item:sheafeq} $p$ is a $G$-equivariant map $p \colon E\lra X$ of spaces with continuous $G$-actions,
\item \label{item:sheafab} $(E,p)$ is a sheaf space of $R$-modules,
\item \label{item:sheafcomb} each map $g \co p^{-1} (x) \lra p^{-1} (g x)$ is a map of $R$-modules for every $x\in X,g\in G$.
\end{enumerate}
We will write this as either the pair $(E,p)$ or simply as $E$.
A map $f \co (E,p) \to (E',p')$ of $G$-sheaves of $R$-modules over $X$ is a 
$G$-equivariant map $f \co E \to E'$ such that $p'f=p$ 
and $f_x \co E_x \to E'_x$ is a map of $R$-modules for each $x \in X$.  
\end{definition}
Note that Points (\ref{item:sheafeq}) and (\ref{item:sheafab}) give a map of
sets for Point (\ref{item:sheafcomb}),
but they do not imply that it is a map of $R$-modules.
This last point is the compatibility between the equivariant structure and the 
sheaf structure.

Similar to Theorem \ref{thm:eqsheafequiv}, these two definitions of equivariant sheaves are equivalent.
See Sugrue \cite[Theorem 4.1.5]{sugruethesis} and K{\k{e}}dziorek \cite[Lemma 7.0.3]{kedziorekthesis}
for the proof. 

\begin{theorem}\label{thm:eqsheafequivRmod}
The definitions of $G$-sheaves of sets from Definition \ref{defn:rmodulecocyclesheaf} and 
Definition \ref{defn:eqsheafRmod} are equivalent.
\end{theorem}

The earlier results on $G$-sheaves of sets hold in this new setting.  
Lemma \ref{lem:discretestalks} implies that the stalk of a $G$-sheaf $F$ at $x \in X$ is a 
discrete $R[\stab_G(x)]$-module. 
Similarly, the analogue of Corollary \ref{cor:sectionsaresections} holds. 
This corollary implies the existence of a $G$-equivariant global zero section. 
That is, for $F$ a $G$-sheaf of $R$-modules over a profinite space $X$, with sheaf space $E$, 
the element $0 \in F(X)$ must correspond to a continuous $G$-equivariant section $0 \co X \to E$. 
Proposition \ref{prop:compactdiscrete} and Corollary \ref{cor:sectionquot} imply that any section is locally sub-equivariant and
$F(U)$ is a discrete $R[\stab_G(U)]$-module whenever $U$ is a compact and open subset of $X$.

The definition of a $G$-presheaf is easily adapted to the case of $R$-modules. 

\begin{definition}\label{def:Gpresheafrmod}
For $X$ a profinite $G$-space with open-closed basis
$\mcB$, a \emph{$G$-equivariant presheaf of $R$-modules over $\mcB$} is 
a functor 
\[
F \co \mcO_G(\mcB)^{op} \lra R \leftmod
\]
such that the action of $\mcO_G(\mcB)^{op}(U,U)=\stab_G(U)$ on $F(U)$ is 
continuous. 

A map of $G$-presheaf of $R$-modules over $\mcB$ is a natural transformation.
\end{definition}

The analogues of Theorems \ref{thm:equivariantsheafspace} and  \ref{thm:equisheafifyfunctor}
hold in the case of $R$-modules.

\begin{example}\label{ex:pointconstantsheafrmod}
As with Examples \ref{ex:onepointsheafsets} and \ref{ex:constantsheafset},  
if we let $X$ denote the trivial one point $G$-space, then a $G$-sheaf of $R$-modules 
over $X$ is a discrete $R[G]$-module.

Let $M$ be a discrete $R[G]$-module (equipped with the discrete topology where needed), then the constant
\emph{equivariant constant sheaf} at $M$ over $X$ is defined by the projection onto the $X$ factor, 
which is a local homeomorphism
\[
p \co M \times X \lra X.
\]
The $G$-action on $M \times X$ is the diagonal action
\[
g(m,x) = (gm, gx)
\]
which is continuous as $M$ is discrete.
\end{example}

As in the non-equivariant setting, the constant sheaf functor is
part of an adjunction with the global sections functor.
The simplest proof makes use of our definition of equivariant presheaves.

\begin{lemma}
Let $X$ be a profinite $G$-space, then the constant sheaf functor $\const$ is 
left adjoint to the global sections functor
\[
\adjunction{\const}{\Gmoddisc{R}{G}}{\Reqsheaf{G}{X}{R}}{\Gamma(X,-).}
\]
\end{lemma}
\begin{proof}
By Corollary \ref{cor:sectionquot}, the global sections of an equivariant sheaf 
are a discrete $R[G]$-module, so the right adjoint lands in the desired category. 

The adjunction argument is analogous to the non-equivariant case. 
The constant equivariant sheaf is the sheafification of 
the constant equivariant presheaf. The constant equivariant presheaf functor 
is left adjoint to the global sections functor. 
\end{proof}

Our next task is to show that the category of $G$-sheaves of $R$-modules
is an abelian category. We start by defining an abelian group 
structure on $\hom(-,-)$, the set of 
maps of $G$-sheaves of $R$-modules. 

\begin{definition}\label{def:sheafaddition}
Let $(E,p)$ and $(F,p^{\prime})$ be two $G$-sheaves over $X$. Then we have:
\begin{enumerate}
\item A zero morphism $0 \colon E\lra F$, given by the map $p \colon E \to X$ followed 
by the global zero section $X \to F$. 

\item If $f,f'\in \hom(E,F)$ we define $f-f'$ to be the morphism where:
\begin{align*}
f+f' \colon E&\lra F\\s_x&\longmapsto f(s_x)-f'(s_x) \in F_x.
\end{align*}
\end{enumerate}
\end{definition}

\begin{proposition}
Let $E$ and $F$ be $G$-sheaves of $R$-modules. 
If $f,f'\in \hom(E,F)$ then $f-f'$ is an element of $\hom(E,F)$ and 
the zero map acts as a unit. 

Hence, the category of $G$-sheaves of $R$-modules is additive. 
\end{proposition}
\begin{proof}
Observe that the map $f-f'$ is given by the composition 
$m\circ \left(f \pi f'\right)\circ\iota$ where:
\[
E\pi E  = \left\lbrace (e,e^{\prime})\in E\times E\mid p(e)=p(e^{\prime})\right\rbrace  
\]
\[
\begin{array}{rclrcl}
\iota \colon E &\lra&  E\pi E, & e  &\longmapsto & (e,e) \\
f \pi f' \colon E \pi E  &\lra&  E\pi E, &  (e,e')  &\longmapsto & (f(e),f'(e')) \\
m \colon E\pi E  &\lra&    E, &  (e,e^{\prime})  &\longmapsto & e-e^{\prime}.
\end{array}
\]
The maps $\iota$ and $m$ are continuous, see \cite[Section 2.5]{tenn}.
Since $f \pi f'$ is continuous, it follows that $f-f'$ is a continuous map.

We see that $f-f'$ is a $G$-equivariant map of sheaf spaces as
\begin{align*}
p^{\prime}((f-f')(s_x)) &= p^{\prime}(f(s_x)-f'(s_x))=x=p(s_x) \\
g((f-f')(s_x))&= g(f(s_x)-f'(s_x))=f(g(s_x))-f'(g(s_x))=(f-f')(g(s_x))
\end{align*}
for all $g \in G$. 
It follows that the zero map is a unit for this addition and $\hom$ is bilinear,
see also \cite[Section 3.2]{tenn}. This implies that the category is additive.
\end{proof}

\begin{theorem}\label{Gsheafabelian}
The category of $G$-sheaves of $R$-modules over $X$ is an abelian category.
\end{theorem}
\begin{proof}
The category of $G$-sheaves of $R$-modules is an additive category via the addition on 
$\hom$ defined earlier. 
The constructions of (finite) limits and colimits is given in 
Constructions \ref{sheafcolim} and \ref{limconstruct}. 
From that work it is clear that finite product and finite coproducts 
agree. Moreover, kernels and cokernels exist. 
It remains to prove that the category is normal, this follows from Tennison \cite[Theorem 2.4.13]{tenn}.
\end{proof}

\section{Colimits and limits of equivariant sheaves and presheaves}\label{sec:colimitssheaves}

We construct colimits and limits of $G$-equivariant presheaves of $R$-modules
then apply this to construct colimits and limits of  $G$-sheaves of $R$-modules.
Point set models for the constructions can be found in  \cite[Sections 4.2 and 4.4]{sugruethesis}.
As every limit is an equaliser of products, it suffices to construct finite limits and 
infinite products. 
We do this separately, as infinite products are much more difficult to construct than finite limits. 

To distinguish presheaf products from sheaf products, we use a superscript $P$ and $S$.
\begin{lemma}\label{lem:Reqpresheafcolimitslimits}
The category of $G$-presheaves of $R$-modules has all colimits, given by taking the termwise colimit. 
Finite limits are constructed termwise. 

If $\{F_i\}_{i \in I}$ is a set of $G$-presheaves of sets over $\mcB$, 
with $I$ an indexing set, then on $U \in \mcB$
\[
\Big( {\prod_{i \in I}}^P F_i \Big) (U) = \disc \prod_{i \in I} F_i(U)
\]
is the infinite product in the category of $G$-presheaves of sets. 
\end{lemma}
\begin{proof}
The most important case is that of infinite products.
Given maps $f_i \colon E \to F_i$ of $G$-presheaves of $R$-modules
we have a map of $R$-modules
\[
f \co E(U) \lra \prod_{i \in I} F_i(U).
\]
By Corollary \ref{cor:sectionquot}, $E(U)$ is a discrete $\stab_G(U)$-module, so 
the image of $f$ must be in the discrete part of the product giving a map
\[
f' \co E(U) \lra \disc \prod_{i \in I} F_i(U).
\]
The monomorphism
\[
\disc \prod_{i \in I} F_i(U) \lra \prod_{i \in I} F_i(U)
\]
implies that composing $f'$ with the projection maps recovers the $f_i$. 
Hence the map $f'$ is unique. 
\end{proof}

The sheafification functor allow us to construct colimits of diagrams of $G$-sheaves of $R$-modules
in terms of the presheaf colimit.

\begin{construction}\label{sheafcolim}
Suppose we have a diagram $F^i$ of $G$-sheaves of $R$-modules (with maps omitted from the notation). 
We may construct the colimit $F$ of $F^i$ in the category of $G$-presheaves of $R$-modules
as earlier. 

The sheaf $\sheafify F$ is the colimit of the diagram $F^i$ in the 
category of $G$-sheaves of $R$-modules. 
As with the non-equivariant case, the stalks of a colimit are the colimits of the stalks. 
\end{construction}

Finite limits of sheaves of $R$-modules over a base space $X$ may be constructed 
by taking the limit in the category of sheaves of sets, then using the induced $R$-action on the stalks. 
Since the limit is finite, the stalks of the limit are exactly the limits of the stalks, so this $R$-action 
is well-defined. 
If $E_i$ is a diagram of sheaves (with maps omitted) and $E$ the limit, then the sheaf space
of $E$ is given by taking the limit of the diagram $E_i$, in the category of spaces over $X$. 
To illustrate, consider the case of a product of two sheaves of $R$-modules over $X$, 
$p \co E \to X$ and $p' \co E' \to X$. The sheaf space of the product is the pullback diagram 
given by $p$ and $p'$, with limit
\[
E \underset{X}{\times} E'  = \{ (e,e') \mid p(e)=p'(e') \}
\]
whose map to $X$ is given by $(e,e') \mapsto p(e)=p'(e')$. 

\begin{construction}\label{limconstruct}
Finite limits of $G$-sheaves of $R$-modules can be constructed following a similar process 
to the non-equivariant case, 
but now the sheaf space of the limit is constructed in the category of $G$-spaces over $X$. 

Equally, we can use Definition \ref{defn:rmodulecocyclesheaf} as our starting point. 
Since the pullback functor of Definition \ref{defn:pullpushfunctors}
preserves finite limits, a finite limit of $G$-sheaves of $R$-modules 
is given by taking the limit of sheaves of $R$-modules and equipping it with the canonical $G$-action
coming from moving the functors $d_0^*$ and $d_1^*$ past the finite limit. 
\end{construction}

The preceding construction of finite limits does not extend to the infinite case. 
Instead, we give an explicit construction of infinite product of $G$-sheaves
of $R$-modules in Proposition \ref{prop:infproduct}. 
A point-set model can be found in \cite[Section 4.4]{sugruethesis}. 
That model can be compared with 
the point-set construction of finite limits in \cite[Subsection 4.2.2]{sugruethesis},
which makes explicit use of the finiteness assumption.
Formally, the existence of infinite products is already known, as the category of $G$-sheaves
of $R$-modules is a Grothendieck topos by Giraud's theorem.  
See Moerdijk \cite{moer88} and \cite{moer90} for this and related results.

The problem is much easier in the case where $G$ is finite (and hence has the discrete topology). 
With this assumption, the construction of infinite products of non-equivariant sheaves 
can be readily extended to the $G$-equivariant case as the 
continuity of the $G$-action may be verified by looking at the map induced by each $g \in G$. 
In particular, if $G$ is a finite discrete group, then 
infinite products of $G$-sheaves are products taken in 
the category of non-equivariant sheaves with the additional structure of a $G$-action. 
This does not hold when $G$ is allowed to be infinite as the following example demonstrates.

\begin{example}
Let $G=\adic{p}$ be the $p$-adic integers. 
The category of $G$-sheaves of $\bQ$-modules over the 
one-point space is equivalent to the 
category of discrete $\bQ[\adic{p}]$-modules. 
The infinite product in the category of discrete $\bQ[\adic{p}]$-modules is the discretisation
of the infinite product. 
Hence, this must be the product in the category of $G$-sheaves of $\bQ$-modules over the 
one-point space.

The product of non-equivariant sheaves of $\bQ$-modules over the one-point space
is the product of the stalks, equipped with the box topology. 
In this case, the box product is the discrete topology on the infinite product. 
The box product is usually not the discretisation of the infinite product. 
For example, the discretisation of the $\bQ[\adic{p}]$-module
\[
\prod_{n \geqslant 0} \bQ[\bZ/p^n]
\] 
is a proper subset of the product. 
\end{example}

\begin{lemma}\label{lem:infprodalreadysheaf}
Let $X$ be a profinite $G$-space with open-closed basis $\mcB$.
Let $\{F_i\}_{i \in I}$ be a set of $G$-sheaves of $R$-modules.
The presheaf product $F=\prod_{i \in I}^P F_i$ satisfies the sheaf condition as a functor
\[
F \co \mcO_G(\mcB)^{op} \lra R \leftmod.
\]
\end{lemma}
\begin{proof}
Take an open cover of $U \in \mcO_G(\mcB)$ by sets $V_i \in \mcO_G(\mcB)$, with $i \in I$ an indexing set.
As $U$ is compact, there is a finite subcover $V_{i_1}$, \dots, $V_{i_n}$. 
Consider the patching and separation conditions for this subcover.  
Since the subcover is finite, the limit over $j=1, \dots, n$ is finite and commutes with the filtered colimit
defining $\disc$. Thus we have isomorphisms
\[
\lim_{j} F(V_{i_j}) 
=
\lim_{j} \disc \prod_{i \in I} F_i(V_{i_j}) 
\cong
\disc \lim_{j} \prod_{i \in I} F_i(V_{i_j}) 
\cong 
\disc  \prod_{i \in I} F_i(U) = F(U). 
\]
One can check that this implies the patching and separation conditions for the original cover.
\end{proof}

\begin{remark}\label{rmk:sheafkanonly}
We can apply the sheafification functor to $\prod_{i \in I}^P F_i$ from the preceding lemma. 
As the presheaf product is already a sheaf over $\mcB$, all we are doing is 
taking a right Kan extension of this object, extending the domain from 
$\mcO_G(\mcB)$ to $\mcO_G(X)$. Hence, the value of 
$\sheafify \prod_{i \in I}^P F_i$ on some $U \in \mcB$ 
is given by 
\[
\Big(\sheafify {\prod_{i \in I}}^P F_i \Big) (U)
=
\disc \prod_{i \in I} \left( F_i (U) \right). 
\]
\end{remark}

\begin{proposition}\label{prop:infproduct}
The infinite product in $G$-sheaves of $R$-modules is given by the sheafification of the 
underlying product of $G$-presheaves of $R$-modules and is independent (up to natural isomorphism)
of the choice of basis for the category of presheaves. 

That is, let $\{F_i\}_{i \in I}$ be a set of $G$-sheaves of $R$-modules
and use a superscript $P$ for the presheaf product and a superscript $S$ for the sheaf product, 
then  
\[
F= {\prod_{i \in I}}^S F_i \coloneqq \sheafify {\prod_{i \in I}}^P F_i 
\]
is the product in the category of $G$-sheaves of $R$-modules.
\end{proposition}
\begin{proof}
Choose an open-closed  basis $\mcB$ to obtain a suitable category of $G$-presheaves of $R$-modules. 
Let $f_i \co E \lra F_i$ be a map of sheaves for each $i \in I$. 
Then we have a map 
\[
f \co E \lra {\prod_{i \in I}}^P F_i
\]
in the category of $G$-presheaves of $R$-modules (over $\mcB$). 
Moreover, if we compose $f$ with the projection from 
\[
p_i \co {\prod_{i \in I}}^P F_i \to F_i
\]
we recover the map $f_i$ as the presheaf product on $U \in \mcB$ is 
a subset of the product of the $F_i(U)$.   
Applying sheafification gives a map
\[
f' \co E \cong \sheafify E \lra \sheafify {\prod_{i \in I}}^P F_i = F
\]
post-composing $f'$ with $\sheafify p_i$ gives $f_i$.  

We must now show that the map $f'$ is uniquely determined.
By Remark \ref{rmk:sheafkanonly}, the sheafification functor does not change the value 
of the presheaf product on open-closed sets in some basis $\mcB$. 
Hence, on these sets the map $f' \co E \lra F$ is uniquely determined. 
The value of the sheaves $E$ and $F$ on any other open set of $X$
is then determined by the sheaf condition using open sets in the basis $\mcB$
and therefore is also unique up to isomorphism. 

By Remark \ref{rmk:basischoice} we see the construction is (up to natural isomorphism) independent 
of the choice of basis $\mcB$ for the category of $G$-presheaves. 
\end{proof}

\section{Pull backs and push forwards}\label{sec:pbpf}

Recall the definition of pull backs and push forwards from Appendix \ref{Sec:sheavespresheaves}.
We prove that these functors pass to categories of equivariant sheaves of sets. 
These results can also be extended to sheaves of $R$-modules. 
 
\begin{lemma}
Let $f \co X \to Y$ be an equivariant map of profinite $G$-spaces, for $G$ be a profinite group.
If $E$ is a $G$-sheaf space over $Y$, then the \emph{pull back} $f^*(E)$ 
\[
\xymatrix{
f^*(E) \ar[r] \ar[d] &
E \ar[d] \\
X \ar[r]_f & Y
}
\]
is a $G$-sheaf space over $X$.
\end{lemma}
\begin{proof}
The space $f^*(E)$ has a $G$-action and the induced map to $X$ is 
a $G$-equivariant local homeomorphism. 
\end{proof}

Arguing from the non-equivariant case, for $x \in X$, there is an isomorphism of stalks
\[
f^*(E)_x \cong E_{f(x)}.
\]
Moreover, if $f$ is open, $f^*(E)(U) = E(f(U))$.

\begin{lemma}
Let $f \co X \to Y$ be an equivariant map of profinite $G$-spaces, for $G$ be a profinite group.
If $F$ is a $G$-sheaf of sets over $X$, then the \emph{push forward sheaf} $f_*(F)$
is a $G$-sheaf of sets over $X$. 
\end{lemma}
\begin{proof}
By definition, $f_*(F)$ is the sheafification of the functor
\[
\mcO(Y)^{op} \overset{f^{-1}}{\lra} \mcO(X)^{op} \overset{F}{\lra} \sets
\]
and hence is a sheaf of sets over $X$. 
We show that on compact-open sets, the group action is discrete. 
The results of Section \ref{sec:equivpresheaves} will then imply that 
$f_*(F)$ is a $G$-sheaf of sets over $X$.

Take $K \subseteq Y$ an open-closed set, which is therefore compact.
The set $f^{-1} K \subseteq X$ is also open-closed and hence compact. 
Thus, $F(f^{-1} K)$ is a discrete $\stab_G (f^{-1} K)$-set by 
Corollary \ref{cor:sectionquot}. 
As $f$ is $G$-equivariant,  $\stab_G (K) \subseteq \stab_G (f^{-1} K)$, 
so the $\stab_G (K)$-action on $f_*(F)(K) = F(f^{-1} K)$ is also discrete.
\end{proof}

\begin{proposition}
Let $f \co X \to Y$ be a $G$-equivariant map of profinite $G$-spaces, for $G$ a profinite group.
The functors $f^*$ and $f_*$ form an adjunction
\[
\adjunction{f^*}{\seteqsheaf{G}{Y}}{\seteqsheaf{G}{X}}{f_*.}
\] 
\end{proposition}
\begin{proof}
Write the left action maps of $G$ as 
$g \co X \to X$ and $g \co Y \to Y$, using context to distinguish them.
These maps are homeomorphisms, hence $g^*=g_*$. 

Let $E$ be a $G$-sheaf of sets over $X$. Then the 
$G$-action on $E$ defines, and is defined by, maps $E \to g^*E=g_*E$. 
Similarly, 
for $F$ a $G$-sheaf of sets over $Y$, the 
$G$-action on $F$ defines, and is defined by, maps $F \to g^*F=g_*F$. 

Assume that we have a $G$-map of $G$-sheaves $F \lra f^*E$. 
Then we have a commutative square as below-left, with its adjoint
square below-right.
\[
\xymatrix{
F \ar[d] \ar[r] & f^*E \ar[d]
&&
f_* F \ar[d] \ar[r] & E \ar[d] \\
g^* F \ar[r] & g^* f^*E = f^* g^* E  &&
g_* f_*  F=f_* g_* F \ar[r] & g_* E  \\
}
\]
Since the right hand square above commutes for each $g \in G$, 
the map $f_* F \lra E$ is a $G$-map as well.  The converse is similar.
\end{proof}

\begin{remark}\label{rmk:equivariantextres}
As with Definition \ref{defn:extres}, when $f \colon X \to Y$ 
is the inclusion of a $G$-subspace, we call $f^* F$ the \emph{restriction}. 
Similarly, we call $f_*$ \emph{extension by zero}
when $f \colon X \to Y$ is the inclusion of closed $G$-subspace $X$.
In this case, the stalks of $f_* F$ are zero outside $X$
and $f^* f_* F$ is isomorphic to $F$. 
\end{remark}

\section{Equivariant skyscraper sheaves}\label{sec:adjunctionsexamples}

In this section we construct an equivariant analogue of a skyscraper 
sheaf -- a sheaf with only one non-zero stalk. 
When working with $G$-sheaves over a profinite $G$-space $X$, we see that 
if a stalk is non-zero at $x$, it must be non-zero at all points of the form $gx$, for $g \in G$.
This is an instance of how in equivariant settings, one should consider $G$-orbits 
in the place of points. 
We conclude that an equivariant skyscraper should be a sheaf which is non-zero on precisely one orbit. 
We work towards this definition, starting by studying sheaves on transitive $G$-spaces
(that is, spaces homeomorphic to a homogeneous space, $G/H$). 
As previously, the group $G$ will be profinite throughout this section. 

\begin{example}\label{transitivesheafex}
Let $H$ be a closed subgroup of $G$ and 
$M$ a  discrete $R[H]$-module. 
Then the map
\[
G \times_{H} M \lra G/H, \quad \quad [g,m] \mapsto gH
\]
defines a $G$-sheaf of $R$-modules over $G/H$. To see that it is a local homeomorphism, 
choose $[e,m]$  in $G \times_{H} M$. Then as $M$ is a discrete $H$-module, there is 
some open normal subgroup of $H$ which fixes $m$. Without loss of generality, we may take that subgroup to be 
of the form $N \cap H$ for $N$ an open normal subgroup of $G$. We want to find an open set containing
$[e,m]$ which is homeomorphic to some open subset of $G/H$.

Let $q \co G \times M \to G \times_H M$ be the quotient map. 
Consider a basic open subset of the domain, $K \times \{l\}$ ($K$ an open subgroup of $G$, $l \in M$).
Since
\[
q^{-1} q (K \times \{l\} ) = K \times (K \cap H)\{l\}
\]
is also an open set of $G \times M$, we see that $q$ is an open map. 
It follows that $q(N \times {m})$ is an open set of  $G \times_H M$.

We now claim that $q(N \times {m})$ is an open set upon which $p$ induces a homeomorphism to the open set 
$NH/H \subset G/H$.
That this restriction of $p$ is a bijection can be checked directly, indeed the domain is 
homeomorphic to $N/N \cap H$. As the sets are all compact and Hausdorff, 
this restriction of $p$ is a homeomorphism as claimed.
\end{example}


As with equivariant vector bundles, requiring that the base space be transitive
is a very strong restriction, as the following lemma makes clear.

\begin{lemma}\label{lem:gsheafghishmod}
For $H$ a closed subgroup of $G$, there is an equivalence of categories between
$G$-sheaves of $R$-modules over $G/H$ and 
discrete $R[H]$-modules. This equivalence is realised by the adjunction
\[
\adjunction{G \times_{H} (-)}{\Gmoddisc{R}{H}}{\Reqsheaf{G}{G/H}{R}}{(-)_{eH}}
\]
where $G \times_{H} -$ sends a discrete $R[H]$-module $M$ to the
sheaf $G \times_{H} M \lra G/H$ of Example \ref{transitivesheafex}.
\end{lemma}
\begin{proof}
Let $M$ be a discrete $R[H]$-module and $E$ a $G$-sheaf of $R$-modules 
over $G/H$. Consider a map of sheaves $f \co G \times_{H} M \lra E$. 
As a map of $G$-spaces, $f$ is uniquely determined by its restriction to 
$(G \times_H M)_{eH} \cong M$. Since $f$ is also a map over $G/H$, it follows that 
the image of this restriction must be in $E_{eH} \subseteq E$. Thus, $f$ is uniquely determined by the map
\[
f_{eH} \co M \lra E_{eH}
\]
and we have shown that the adjunction exists. The unit is the isomorphism
\[
M \to (G \times_H M)_{eH}.
\]
The counit is the map $G \times_H E_{eH} \to E$, which is an isomorphism on 
the stalk at $eH$. As the base space  $G/H$ is transitive, the counit is an isomorphism on all stalks
and \cite[Theorem 3.4.10]{tenn} implies that the counit is an isomorphism of sheaves. 
\end{proof}

We can relate this example and lemma to general sheaves by fixing a specific point $x$ in the base space
and considering the orbit $O=Gx$ of $x$. 

\begin{lemma}\label{restconst}
Let $O=Gx$ be the orbit of a point $x$ in a $G$-space $X$
and $E$ be a $G$-sheaf of $R$-modules over $X$. 
There is a natural isomorphism of $G$-sheaves over $O$
\begin{align*}
\psi \co G\underset{\stab_G(x)}{\times}E_x \cong E_{|_{O}}.
\end{align*}
\end{lemma}

We combine this result with the extension by zero and restriction functors 
of Remark \ref{rmk:equivariantextres} to give a functor 
from discrete $R[H]$-modules to $G$-sheaves over a profinite $G$-space $X$. 
We can use extension and restriction (rather than the more general push forward or pull back
functors) as orbits are closed.

\begin{example}\label{Gskyscraper}
Let $O=Gx$ be an orbit of a profinite $G$-space $X$, with $i \co O \lra X$ the inclusion. 
For a discrete 
$R\left[\stab_G(x)\right]$-module $M$ 
there is a sheaf over $X$ 
\[
i_* G\underset{\stab_G(x)}{\times} M
\]
where $i_*$ is the extension by zero functor.
We call this the \emph{equivariant skyscraper} of $M$ over $O$. 
The stalk of this sheaf at $x$ is $M$, the stalks at other points in $O$
are isomorphic to $M$ as $R$-modules and the stalks at $z \notin O$ are zero. 
\end{example}

Thus we have a definition of an equivariant generalisation of a skyscraper 
sheaf. If we think of this construction as a functor, 
we find it has a right adjoint given by taking a stalk of the sheaf. 

\begin{proposition}\label{equiadjunct}
Let $O=Gx$ be an orbit of a profinite $G$-space $X$, with $i \co O \lra X$ the inclusion. 
There is an adjunction 
\[
\adjunction
{(-)_{x}}
{\Reqsheaf{G}{X}{R}}
{\Gmoddisc{R}{\stab_G(x)}}
{i_* G\underset{\stab_G(x)}{\times} (-) .}
\]
\end{proposition}
\begin{proof}
Since the restriction functor does not alter stalks, the right adjoint is the composite
of the equivalence of Lemma \ref{lem:gsheafghishmod} and the restriction functor of 
Remark \ref{rmk:equivariantextres}.
\end{proof}

\section{Godemont resolutions of  equivariant sheaves}

A particularly useful property of skyscraper sheaves over a space $X$ is that they contain a 
set of enough injectives for the category of (non-equivariant) sheaves over $X$. 
Hence, they can be used to give a canonical injective resolution of a sheaf, often known as the 
Godemont resolution, see Bredon \cite[pp. 36, 37]{bredonsheaf}.
We replicate this for equivariant sheaves. 
Our aim is to show that equivariant skyscraper sheaves
give enough injectives for $G$-sheaves of $R$-modules over $X$.
We first see that the equivariant skyscraper functor preserves injective objects. 
We remind the reader that $R$ is a unital ring, $G$ is a profinite group and $X$ is a profinite $G$-space. 

\begin{proposition}\label{equivinjshf}
Let $x$ in $X$. 
If $A$ is an injective discrete $R[\stab_G(x)]$-module,
then the $G$-sheaf 
\[
i_* G\underset{\stab_G(x)}{\times} A
\] 
is injective in the category of $G$-sheaves of $R$-modules over $X$. 
\end{proposition}
\begin{proof}
As the left adjoint of Proposition \ref{equiadjunct} preserves monomorphisms, 
the right adjoint preserves injective objects.
\end{proof}

For any $G$-sheaf of $R$-modules we may
construct a monomorphism into an injective sheaf. 
Non-equivariantly this is usually done by taking a product 
of skyscraper sheaves indexed over the points of the base space. 
Equivariantly, we use equivariant skyscraper sheaves and index over
\emph{orbits} of the base space.

\begin{theorem}\label{thm:enoughinj}
The injective equivariant skyscraper sheaves provide enough injectives for 
the category of $G$-sheaves of $R$-modules over $X$.
\end{theorem}
\begin{proof}
Let $E$ be a $G$-sheaf of $R$-modules over $X$.
For each orbit $O$ of $X$, pick an element $x_O \in O$. 
By Lemma \ref{restconst} there is an isomorphism of sheaves over $O$ 
\begin{align*}
G\underset{\stab_G(x_O)}{\times}E_{x_O} \cong E_{|_{O}}.
\end{align*}
For each $x_O \in X$, we use Lemma \ref{lem:rgmodinjectives} to obtain
a monomorphism into an injective  discrete $R[\stab_G(x)]$-module
\[
E_{x_O} \lra I_{x_O}.
\]
By Tennison \cite[Theorem 3.3.5]{tenn}, this induces a monomorphism of $G$-sheaves over $O$
\[
E_{|_{O}} \cong G\underset{\stab_G(x_O)}{\times} E_{x_O}
\lra
G\underset{\stab_G(x_O)}{\times} I_{x_O}
\] 
into an injective $G$-sheaf over $O$.
By the extension--restriction adjunction, this gives a 
map of $G$-sheaves over $X$
\[
E \lra i^* G\underset{\stab_G(x_O)}{\times} I_{x_O} = I_O.
\]
This map is induces a monomorphism on stalks at every $y \in O$.

We claim the product of these maps
\[
E \lra {\prod_{O}}^S I_O
\]
is a monomorphism into an injective sheaf. 
To see it is a monomorphism we compose with the 
projection to $I_O$ and check stalkwise for points $y \in O$.
By Proposition \ref{equivinjshf}, the codomain is injective,
so we have finished our claim and the proof. 
\end{proof}

\begin{definition}\label{godeequi}\index{equivariant Godement resolution}
If $E$ is a $G$-sheaf of $R$-modules over a profinite $G$-space we define the 
\emph{equivariant Godement resolution} as follows.

Let $\delta_0 \co E \lra I_0(E)$ be the monomorphism into an injective sheaf constructed in the proof of 
Theorem \ref{thm:enoughinj} applied to the sheaf $E$.
\begin{align*}
\xymatrix{0\ar[r]&E\ar[r]^{\delta_0}&I^0(E)\ar[d]^p\ar[r]&I^0(\coker\delta_0)=I^1(E)\ar[r]&\ldots\\&&\coker\delta_0\ar[ur]_{\delta_1}}
\end{align*}
Where $\delta_1$ and $I^1(E)$ come from Theorem \ref{thm:enoughinj} applied to the sheaf $\coker\delta_0$.
\end{definition}

As $\delta_1$ is a monomorphism, we see that the sequence is exact at $I^0(E)$.
Inductively, we see that this is an injective resolution of $E$.
This resolution is particularly sensitive to the stalks at isolated points
and was used in \cite[Chapter 7]{sugruethesis} to give calculations of the injective dimension of 
the category of rational $G$-sheaves over a profinite $G$-space $X$. 
The authors will revisit this point in \cite{BSspectra}.




\section{Weyl sheaves over spaces of subgroups}

The second author used the theory of equivariant sheaves to 
construct an equivalence of categories between the category of rational $G$-equivariant Mackey functors with
a full subcategory of the category of $G$-sheaves of $\bQ$-modules. 
This equivalence requires $G$ to be profinite and 
can be found in \cite[Chapter 5]{sugruethesis} and \cite{BSmackey}. 
In this section we examine this subcategory, known as the category of Weyl-$G$-sheaves. 
We start by introducing the space of closed subgroups of a profinite group $G$, 
which is the base space used in the equivalence. This section overlaps slightly with \cite[Section 2.3]{BSmackey}, see Appendix \ref{Sec:profinitestuff} for a brief recap of 
profinite spaces. 

\begin{definition}
For $G$ a profinite group, let $\sub G$ denote the set of closed subgroups of $G$.
For $N$ an open normal subgroup of $G$, we define surjective maps 
\[
p_N \co \sub G \lra \sub G/N \quad \quad K \longmapsto KN/N.
\]
\end{definition}. 

We use the maps $p_N$ to put a topology on $\sub G$. This gives open sets
\[
p_N^{-1} \{ KN/N \} = O(N,NK) = \left\lbrace L \in \sub G\mid NL= NK \right\rbrace
\]
for $K$ a closed subgroup of $G$. These sets form a basis $\mcB_{\sub G}$ for the topology.
Indeed, if $N$ and $M$ are open normal subgroups and $K$ and $L$ closed subgroups, 
then 
\[
O(N \cap M, (N \cap M) C ) \subseteq O(N,NK) \cap O(M, ML)
\]
for any $C \in O(N,NK) \cap O(M, ML)$. Moreover, 
if $N \leqslant M$, then $O(N,NK)\subseteq O(M,MK)$.

\begin{examples}
The following examples can be found as 
\cite[Proposition 2.5 and Example 3.2]{GartSmith10}
and \cite[Page 3]{GSclassify}.
The space $\sub \adic{p}$ can be described as the set 
\[
P=\{0\} \cup \{ 1/n \mid n \in \bN \} \subset \bR.
\]
If $p_1, \dots, p_n$ are distinct primes then 
\[
S \left( \prod_{i=1}^n \adic{p_i} \right) = P^n.
\]
The space of subgroups of $(\adic{p})^{\times n}$ is the 
Pe\l{}czy\'nski space.
\end{examples}

\begin{lemma}\label{lem:intersections}
Let $G$ be a profinite group and $K$ a closed subgroup. Then
\[
K=\bigcap_{N \opennormalsub G} NK  \quad \textrm{ and } \quad 
\{ K \} = \bigcap_{N \opennormalsub G} O(N, NK).
\]
\end{lemma}
\begin{proof}
The intersection contains $K$, so assume $g \in G$ is in the intersection.
For any open normal subgroup $N$, we may write $g=nk$ for $n \in N$ and $k \in K$.
Hence $n = g k^{-1} \in N \cap gK$.
The sets $N \cap gK$ are non-empty and satisfy the finite intersection property for varying $N$,
as $N$ is an open subgroup. By compactness of $G$, we have that
\[
\emptyset \neq \bigcap_{N \opennormalsub G} N \cap gK \subseteq \bigcap_{N \opennormalsub G} N = \{ e\}.
\]
Hence, for any $N$, $e \in N \cap gK \subseteq gK$ for any $N$.
We have shown that $g \in K$.

For the second equality, consider some closed subgroup $L$ such that $NL=NK$ for all $N$. 
Then 
\[
K=\bigcap_{N \opennormalsub G} NK = \bigcap_{N \opennormalsub G} NL=L. \qedhere
\]
\end{proof}

The space $\sub G$ has a continuous $G$-action given by 
\[
G \times \sub G \lra \sub G \quad \quad (g,K) \longmapsto gKg^{-1}.
\]
The stabiliser of a closed subgroup $K$ of $G$ in $\sub G$ is therefore the normaliser
of $K$ in $G$, $N_G K$. We may also calculate the stabiliser of the basic open sets:
\begin{align*}
\stab_G\left(O(N,NK)\right)=N_G(NK).
\end{align*}


\begin{example}\label{ex:groupring}
We define a $G$-equivariant presheaf $F$ of $R$-modules over $\mcB_{\sub G}$ by
\[
F(O(N,NK))  = R[G/NK]
\]
with structure maps for inclusion and translation as follows.
Since 
\[
O(M,ML) \subset O(N,NK)   \Longleftrightarrow ML \leqslant NK
\]
the inclusion map $i \colon O(M,ML) \subset O(N,NK)$ defines a projection $p \colon G/ML \to G/NK$.  
We define the restriction map $F(i)$ to be the transfer map
\[
F(i) \colon R[G/NK] \lra R[G/ML], \quad \quad xNK \mapsto \sum_{p(yML)=xNK} yML
\]
which sends $gNK$ to the sum of elements in its preimage under $p$.
Given a translation map $g \colon O(N,NK) \to O(N,NgKg^{-1})$ we define
\[
F(g) \colon R[G/NgKg^{-1}] \lra R[G/NK], \quad \quad xNgKg^{-1} \mapsto g^{-1}xgNK.
\]

The restriction maps are equivariant in the sense that 
the maps $F(g)$ and $F(i)$ fit into a commutative square. 
Moreover, if $g \in NK$, the map $F(g)$ is given by 
\[
R[G/NK] \lra  R[G/NK], \quad \quad xNK \mapsto g^{-1} xNK
\]
and hence is continuous, as the left-multiplication action 
of $N_G(NK)$ on the discrete space $G/NK$ is continuous.
Hence, $F$ defines a $G$-equivariant presheaf of $R$-modules as in 
Definition \ref{def:Gpresheafrmod}.

We can use equivariant sheafification, Theorem \ref{thm:equisheafifyfunctor}, 
to construct a $G$-equivariant sheaf of $R$-modules $E$ from $F$. 
By Lemma \ref{lem:intersections}, the stalk of $E$ at a closed subgroup $K$ of $G$
is 
\[
E_K=F_K =\colim_{N \opennormalsub G} R[G/NK].
\]
When $R=T^*$, the dual of a profinite ring $T$, this colimit is the dual 
of the completed group algebra $\llbracket T G \rrbracket$ of \cite[Section 5.3]{rz00}.
\end{example}

\begin{example}\label{ex:groupring2}
As a related example, we can define a $G$-presheaf and $G$-sheaf by taking
\[
F'(O(N,NK))  = R[(G/NK)^L]
\]
and using the same structure maps. The stalk at $K$ is 
\[
F'_K=\colim_{N \opennormalsub G} \bQ[(G/NK)^L]
\]
which we claim is zero unless $L$ is conjugate to a subgroup of $NK$. 
To prove the claim, assume that $L$ is not conjugate to a subgroup of K.
Then we have an inverse system of finite sets with limit the empty set 
\[
\emptyset = (G/K)^L = \ilim_{N \opennormalsub G} (G/NK)^L.
\]
By \cite[Proposition 1.1.4]{rz00} we must have that some $(G/NK)^L$
is empty, hence so is $(G/M K)^L$ whenever $M \leqslant N$. 
The claim follows from a cofinality argument.
\end{example}


\begin{definition}\label{defn:weylsheaf}
A \emph{Weyl-$G$-sheaf} of $R$-modules $E$ is a $G$-sheaf of $R$-modules
over $\sub G$ such that $E_K$ is $K$-fixed and hence 
is a discrete $R[W_G(K)]$-module. 
A map of Weyl-$G$-sheaves is a map of $G$-sheaves of $R$ modules over $\sub G$.
We write this category as $\Rweylsheaf{G}{R}$, with the inclusion into 
$G$-sheaves of $R$-modules over $\sub G$ denoted by $\inc$. 

Let $F$ be a $G$-sheaf of $R$-modules over $\sub G$. 
A \emph{Weyl section} of $F$
is a section $s \co U \to F$ such that for each $K \in U$, $s_K$ is $K$-fixed. 
\end{definition}

\begin{example}
The constant sheaf at an $R$-module $M$ is a Weyl-$G$-sheaf. It should be 
remembered that here we are taking an $R$-module, not an $R[G]$-module. 

Let $K$ be a closed subgroup of $G$, if $A$ is a $K$-fixed discrete $R[N_G K]$-module, 
then the equivariant skyscraper sheaf
\[
i_* G\underset{N_G K}{\times} A 
\]
is a Weyl-$G$-sheaf of $R$-modules that we may call a \emph{Weyl skyscraper sheaf}. 
\end{example}

The definitions directly imply that every section of a Weyl-$G$-sheaf is a Weyl section. 
Weyl sections occur quite often in general $G$-equivariant 
sheaves. We restate the local sub-equivariance of 
sections of sheaves of $\sub G$ and show that $K$-fixed germs 
over $K \in \sub G$ are represented by Weyl sections.
This result and proof also appears as \cite[Proposition 2.21]{BSmackey}.

\begin{proposition}\label{prop:localsubeqweyl}
If $E$ is a $G$-sheaf of $R$-modules over $\sub G$ and $s_K \in E_K$ is fixed by $K$, 
then there there is an open normal subgroup $N$ of $G$ such that $s_K$ is represented by an $NK$-equivariant section:
\begin{align*}
s \colon O(N,NK) \lra E.
\end{align*}
\end{proposition}

\begin{corollary}\label{cor:fixgermweylsection}
Let $E$ be a $G$-sheaf of $R$-modules over $\sub G$.
A $K$-fixed germ in $E_K$ can be represented by a Weyl section. 
\end{corollary}
\begin{proof}
By Proposition \ref{prop:localsubeqweyl}, we have an $NK$-equivariant  representative
\[
s \colon O(N,NK) \lra E
\]
for a $K$-fixed germ $s_K \in E_K$. Let $L \in O(N,NK)$, then $L \leqslant NL = NK$, 
so $s$ is $L$-equivariant. As $L \in \sub G$ is fixed by the action of $L$, 
$s_L$ is $L$-fixed. 
\end{proof}

Using these results we can construct a Weyl-$G$-sheaf from a $G$-equivariant 
sheaf of $R$-modules over $\sub G$. 

\begin{construction}
Let $E$ be a $G$-sheaf of $R$-modules over $\sub G$.
We define a Weyl-$G$-sheaf $\Weyl(E)$. 
The stalk over $K$ is given by the $K$-fixed points of $E_K$
\[
\Weyl(E)_K = E_K^K.
\]
We then equip $\Weyl(E)$ with the subspace topology. 
Equally, we define the topology using only the Weyl sections of 
$E$. As every germ of $\Weyl(E)$ can be represented by a Weyl section, 
we have the local homeomorphism condition.  
\end{construction}

As we are working with sheaves of $R$-modules, we note that the stalks of 
$\Weyl(E)$ will be non-empty, as the zero element is always fixed by the group actions. 

Since the image of a $K$-fixed element of a $G$-space must be $K$-fixed, 
we have the adjunction of the following lemma. 
The last statement of the lemma is equivalent to the statement that
$\inc$ is fully faithful.

\begin{lemma}\label{lem:weyladjunction}
For $G$ a profinite group there is an adjunction
\[
\adjunction{\inc}{\Rweylsheaf{G}{R}}{\Reqsheaf{G}{\sub G}{R}}{\Weyl.}
\]

Moreover, the counit of this adjunction
\[
\inc \Weyl(E) \lra E
\]
is a monomorphism of $G$-sheaves of $R$-modules.
and the unit map 
\[
F \lra \Weyl(\inc F) 
\]
is an isomorphism. 
\end{lemma}

\begin{remark}\label{rmk:weylmackeysheaf}
As mentioned, earlier work of the authors \cite{BSmackey} gives an 
equivalence between rational $G$-equivariant Mackey functors and 
Weyl-$G$-sheaves of $\bQ$-modules
\[
\adjunction{\mackeyfunctor}{\weylsheaf{G}}{\mackey{G}}{\sheaffunctor.}
\]
The functor $\mackeyfunctor$ can in fact take a 
$G$-sheaf of $\bQ$-modules over $\sub G$ as input. 
By \cite[Proposition 5.5]{BSmackey} the construction 
\[
\sheaffunctor \circ \mackeyfunctor (E)
\]
for $E$ a $G$-sheaf (not necessarily Weyl) gives $\Weyl(E)$. 

The construction of a rational Mackey functor $\mackeyfunctor(E)$ from a $G$-sheaf $E$ defines
\[
\mackeyfunctor(E)(H) = E(\sub H)^H.
\]
If $K \in \sub H$, then any $H$-fixed (or equally, $H$-equivariant) section $s \in E(\sub H)$ 
is $K$-equivariant, so $s_K$ is $K$-fixed. 
Hence, the construction of a Mackey functor from a $G$-sheaf 
only makes use of the Weyl sections. 
As a consequence, 
\[
\mackeyfunctor (\Weyl E) = \mackeyfunctor(E)
\]
giving a diagram of functors as below.
\[
\xymatrix{
\Rweylsheaf{G}{\bQ} 
\ar@<0.7ex>[r]^-{\inc}
\ar@<0.7ex>[d]^-{\mackeyfunctor} 
&  
\Reqsheaf{G}{\sub G}{\bQ} \ar@<0.7ex>[l]^-{\Weyl} 
\ar@<+2pt>@/^10pt/[dl]^-{\mackeyfunctor}
\\
\mackey{G}
\ar@<0.7ex>[u]^-{\sheaffunctor}
}
\]
\end{remark}

\begin{theorem}\label{thm:weylabelian}
The category of Weyl-$G$-sheaves of $R$-modules is an abelian category with all
small colimits and limits. 
Limits are constructed by applying $\Weyl$ to the underlying diagram of 
$G$-sheaves of $\bQ$-modules over $\sub G$. 
\end{theorem}
\begin{proof}
Take a colimit diagram of Weyl-$G$-sheaves of $R$-modules. 
Applying $\inc$ and taking the colimit gives a $G$-sheaf
whose stalk at $K \in \sub G$ is the colimit of the stalks of the diagram. 
It follows that the stalk at $K$ of the colimit is $K$-fixed. 
Hence, the colimit is a Weyl-$G$-sheaf. 
Finite limits may be constructed similarly. 

The addition operation for maps of $G$-sheaves
from Definition \ref{def:sheafaddition} applies equally well to Weyl-$G$-sheaves. 
The global zero section of a $G$-sheaf is $G$-equivariant, hence  
a Weyl-$G$-sheaf has a global zero section. This gives a zero map in the set of
maps from one Weyl-$G$-sheaf to another Weyl-$G$-sheaf. 
Arguments similar to those in the proof of Theorem \ref{Gsheafabelian}
show that we have an abelian category. 

It remains to construct (small) infinite products. Let $E_i$ be a set of  
Weyl-$G$-sheaves of $R$-modules and $F$ another Weyl-$G$-sheaf. 
Using a superscript $S$ for the sheaf product we have isomorphisms
\begin{align*}
\Rweylsheaf{G}{R}(F, \Weyl {\prod}^S_i \inc E_i)
\cong &
\Reqsheaf{G}{\sub G}{R}(\inc F, {\prod}^S_i \inc E_i) \\
\cong &
\prod_i \Reqsheaf{G}{\sub G}{R}(\inc F, \inc E_i) \\
\cong &
\prod_i \Rweylsheaf{G}{R}(F, E_i)
\end{align*}
Hence, the product in the category of Weyl-$G$-sheaves is given by $\Weyl \prod^S_i \inc E_i$.
\end{proof}

When working rationally, Remark \ref{rmk:weylmackeysheaf}
and \cite[Subsection 5.3.3]{sugruethesis} 
give another construction of limits via the category 
of rational $G$-Mackey functors. That description has the 
advantage that the product of Mackey functors is constructed objectwise.

We can relate the $(\inc,\Weyl)$-adjunction to equivariant skyscraper sheaves by 
the following commutative diagram of adjunctions (which means the diagram of left adjoints 
commutes up to natural isomorphism and the diagram of right adjoints 
commutes up to natural isomorphism).

\begin{lemma}
For any closed subgroup $K$ of $G$, there is a commutative diagram of adjunctions 
\[
\xymatrix@C+1cm@R+0.5cm{
\Rweylsheaf{G}{R}
\ar@<-0.7ex>[d]_-{\inc} 
\ar@<+0.7ex>[r]^-{(-)_K} 
&
\Gmoddisc{R}{W_G K} 
\ar@<-0.7ex>[d]_-{\inc} 
\ar@<+0.7ex>[l]^-{i_* G\underset{N_G K}{\times} (-)} 
\\
\Reqsheaf{G}{\sub G}{R}
\ar@<-0.7ex>[u]_-{\Weyl} 
\ar@<+0.7ex>[r]^-{(-)_K} 
&
\Gmoddisc{R}{N_G K}. 
\ar@<-0.7ex>[u]_-{(-)^K} 
\ar@<+0.7ex>[l]^-{i_* G\underset{N_G K}{\times} (-)} 
}
\]
\end{lemma}

Since the above mentioned stalk functors preserves monomorphisms, the horizontal right adjoints preserve
injective objects. This fact and the evident analogues of Theorem \ref{thm:enoughinj}
and Definition \ref{godeequi}  
(using discrete $\bQ[W_G K]$-modules in place of discrete $\bQ[N_G K]$-modules)
give the following corollary.

\begin{corollary}
Let $K$ be a closed subgroup of a profinite group $G$. 
If $A$ is an injective discrete $\bQ[W_G K]$-module, then the Weyl skyscraper sheaf 
\[
i_* G\underset{N_G K}{\times} A
\]
is an injective object of Weyl-$G$-sheaves of $R$-modules. 
Hence, the category of Weyl-$G$-sheaves of $R$-modules has enough injectives
and the Godement resolution can be used for Weyl-$G$-sheaves.
\end{corollary}

\section{Sheaves over a diagram of spaces}\label{sec:diagramsheaves}

A profinite space is an inverse limit of finite discrete spaces.
It is logical to ask if this structure can be applied 
to describe the category of sheaves over a profinite space $X$ in terms of 
a diagram of sheaves over the finite spaces from the limit defining $X$. 
Working equivariantly for a profinite group $G$, we further want a finite group action
on each of the finite spaces, with $G$ the limit of those finite groups. 
We provide an equivalence between a class of diagrams of equivariant sheaves 
and $G$-sheaves of $R$-modules over $X$ in Theorem \ref{thm:diagramsheafcolim}.
Analogous statements and proofs hold for equivariant sheaves of sets. 

The starting point is to extend the notion of push forward and pull back sheaves 
of Section \ref{sec:pbpf} to include a change of group. A full treatment of change of group functors would take too long, so we
focus on the case of most interest to the aim of this section. 

\begin{lemma}\label{lem:doublechangefunctor}
Let $N$ be an open normal subgroup of a profinite group $G$ and $p^G \co G \to G/N$ the projection. 
Let $p \co X \to Y$ be a surjective map from a profinite $G$-space $X$ to a $G/N$ space $Y$
such that $p^N(g) p(x) =  p(gx)$ for all $g \in G$ and $x \in X$. 

There is an adjunction 
\[
\adjunction
{\bar{p}^*}
{\Reqsheaf{G/N}{Y}{R}}
{\Reqsheaf{G}{X}{R}}
{\bar{p}_*}
\]
with the functors explained in the following construction.
\end{lemma}

\begin{construction}
Let $D$ be a $G/N$-sheaf of $R$-modules over $Y$. Ignoring the group actions, 
$p^*D$ is a sheaf of $R$-modules over $X$ taking values
\[
U \mapsto D(p U)
\]
for $U$ an open subset of $X$ (as the map $p$ is open).
As $Y$ is a $G/N$-space, $D(p U)$ has an action of $\stab_{G/N}(p U)$. 
As the map $p^G$ sends $\stab_{G}(U) \leqslant G$ to $\stab_{G/N}(p U)$, 
we may let $\stab_{G}(U)$ act on $D(p U)$ via $p^G$. 
This change of group functor (an instance of inflation) 
would normally be called $(p^G)^*$, we omit it from out notation
to avoid confusion with the change of base space functors. 
With this $G$-action, we have a $G$-sheaf over $X$ which we call $\bar{p}^*D$.

We construct the adjoint to $\bar{p}^*$. We will use the fact that for
any open subset $V$ of $Y$, $N$ is a subgroup of the stabiliser of $p^{-1}V$. 
Let $E$ be a $G$-sheaf of $R$-modules over $X$, then the functor 
\[
V \mapsto E(p^{-1}V)^N
\]
for $V$ an open subset of $Y$, defines a $G/N$ sheaf over $Y$.
We denote this sheaf by $\bar{p}_*E$. 

To see that we have an adjunction, one checks the triangle identities. 
The unit of this adjunction is the identity map, the counit at an open set $U$
\[
E(p^{-1}(p U))^N \longrightarrow E(U) 
\]
is given by combining the inclusion of fixed points with the restriction map. 
That the triangle identities hold 
comes from the triangle identity for the inflation-fixed point adjunction and the 
fact that $p(p^{-1} V) = V$ for any open $V$ in $Y$. 
\end{construction}

We now turn to the context of interest to us. 
Fix a cofiltered category $I$ and diagram of finite discrete groups and surjective group homomorphisms
\[
G_\bullet = \{G_i, \phi_{ij}^G \co G_i \to G_j \}
\] 
over $I$. 
Let $G$ be the limit of this diagram, with projection maps $p_i^G$, so that
$\phi_{ij}^G \circ p_i^G = p_j^G$. We let $N_i$ be the kernel of $p_i^G$.
Similarly, fix a 
diagram of finite discrete spaces and surjective maps
\[
X_\bullet = \{X_i, \phi_{ij} \co X_i \to X_j \}
\]
with projections maps $p_i$ such that $X_i$ is a $G_i$-space and the 
following diagram commutes.
\[
\xymatrix{
G_i \times X_i \ar[r] \ar[d]_{\phi_{ij}^G \times \phi_{ij}} &
X_i \ar[d]_{\phi_{ij}} \\
G_j \times X_j \ar[r] &
X_j 
}
\]
It follows that $G$ acts on $X$, the limit of the $X_i$, and the evident diagram of projection maps
and action maps commutes.

\begin{definition}
A \emph{$G_\bullet$-sheaf of $R$-modules over $X_\bullet$} is 
a $G_i$-equivariant sheaf $D_i \lra X_i$ of $R$-modules for each $i$, 
with maps of sheaves of $R$-modules  over $X_i$
\[
\alpha_{ij} \co (\bar{\phi}_{ij})^* D_j \lra D_i
\]
for each $i \to j$ in $I$.  
A map $f_\bullet \co D_\bullet \lra D'_\bullet$ of $G_\bullet$-sheaves of $R$-modules over $X_\bullet$ consists of maps 
$f_i \co D_i \to D_i'$ of $G_i$-sheaves of $R$-modules over $X_i$ that commute with the structure maps
\[
\xymatrix{
(\bar{\phi}_{ij})^* D_j \ar[r]^-{\alpha_{ij}} \ar[d]_-{\phi_{ij}^* f_j} & 
D_i \ar[d]^-{f_i} \\
(\bar{\phi}_{ij})^* D'_j \ar[r]_-{\alpha'_{ij}} & 
D'_i. \\
}
\]
We write $\Reqsheaf{G_\bullet}{X_\bullet}{R}$ for this category. 
\end{definition}

Note the contravariance introduced in the definition, the maps
of sheaves reverse the maps in the diagram $I$. 
It can be useful to write out how the functor $(\bar{\phi}_{ij})_*$ acts. 
For $U$ an open subset of $X_j$, the sheaf $D_i$ is defined by
\[
(\bar{\phi}_{ij})_* D_i (U) =  D_i (\phi_{ij}^{-1} U)^{N_j/N_i}.
\]
Moreover, $N_j/N_i$ is a subgroup of $G/N_i=G_i$ which is in the stabiliser of $\phi_{ij}^{-1} U$.

We can construct a $G_\bullet$ sheaf of $R$-modules over $X_\bullet$ from any $G$-sheaf over $X$
from push forward sheaves.

\begin{construction}
Let $E$ be a $G$-sheaf of $R$-modules over $X$. For each $i$ we have a push forward sheaf 
$E_i = (\bar{p}_i)_* E$ over $X_i$. For $i \to j$ in $I$, we have an equality
\[
E_j = (\bar{p}_j)_* E \overset{=}{\lra} (\bar{\phi}_{ij})_* (\bar{p}_i)_* E  = (\bar{\phi}_{ij})_* E_i
\]
of sheaves over $X_j$, using the equality $((-)^{N_i})^{N_j/N_i} = (-)^{N_j}$ 
and composition of the non-equivariant push forward functors.  
Taking the adjoint gives a map of sheaves over $X_i$
\[
\alpha_{ij} \co (\bar{\phi}_{ij})^* E_j \lra E_i.
\]
These sheaves and maps define a $G_\bullet$-sheaf of $R$-modules $\bar{p}_\bullet E$ over $X_\bullet$.
\end{construction}

In the reverse direction, we can construct a $G$-sheaf of $R$-modules over $X$ from a 
$G_\bullet$-sheaf of $R$-modules over $X_\bullet$ by taking a colimit of pullback sheaves. 

\begin{construction}
Given a $G_\bullet$-sheaf of $R$-modules $D_\bullet$ over $X_\bullet$, we have a sheaf
$(\bar{p}_i)^* D_i$ for each $i \in I$. For $i \to j$ in $I$, we have maps
\[
(\bar{p}_j)^* D_j = (\bar{p}_i)^* (\bar{\phi_{ij}})^* D_j \lra (\bar{p}_i)^* D_i
\]
of $G_j$-sheaves of $R$-modules over $X_j$. Define $D= \colim_i (\bar{p}_i)^* D_i$.
\end{construction}

The functor $\bar{p}_\bullet$ is part of an adjunction and is fully faithful, that is, 
$G$-sheaves over $X$ are a full subcategory of $G_\bullet$-sheaves over $X_\bullet$.
We can restrict that adjunction to an equivalence of categories. 

\begin{theorem}\label{thm:diagramsheafcolim}
There is an adjunction
\[
\adjunction{\colim_i (\bar{p}_i)^*}{\Reqsheaf{G_\bullet}{X_\bullet}{R}}{\Reqsheaf{G}{X}{R}}{\bar{p}_\bullet.}
\]
whose counit is an isomorphism. This adjunction restricts to an equivalence of categories
\[
\adjunction{\colim_i (\bar{p}_i)^*}{\colim \Reqsheaf{G_\bullet}{X_\bullet}{R}}{\Reqsheaf{G}{X}{R}}{\bar{p}_\bullet}
\]
where the left hand category is the full subcategory of $\Reqsheaf{G_\bullet}{X_\bullet}{R}$
of those sheaves $D_\bullet$ of $R$-modules over $X_\bullet$ whose adjoint structure maps 
\[
\bar{\alpha}_{ij} \co D_j \lra (\bar{\phi}_{ij})_* D_i
\]
are isomorphisms. 
\end{theorem}

\begin{proof}
That the two constructions fit into an adjunction is similar to the colimit--constant functor adjunction combined with 
the change of base space \emph{and} group functors of Lemma \ref{lem:doublechangefunctor}. 
The counit statement is mentioned (non-equivariantly) by Goodwillie and Lichtenbaum in the proof of \cite[Lemma 3.4]{GL01}.
Following that reference, we check that the counit is an isomorphism on stalks.  

Take some $x \in X$ and $E$ a $G$-sheaf of $R$-modules over $X$. 
We show the counit is an isomorphism by considering the stalk of $E$ at $x$, as follows.
\begin{align*}
(\colim_i (\bar{p}_i)^* (\bar{p}_i)_*E)_x
& =
\colim_i ((\bar{p}_i)^* (\bar{p}_i)_*E)_x
=
\colim_i ((\bar{p}_i)_*E)_{p_i x} \\
& =
\colim_i ((\bar{p}_i)_*E) ( \{ p_i x \})
=
\colim_i E (p_i^{-1} \{ p_i x \})^{N_i} 
\overset{\cong}{\lra} 
E_x. 
\end{align*}
We use two facts to see that the final map is an isomorphism.
The first is that $x$
has a neighbourhood basis given by the open sets $p_i^{-1} \{ p_i x \}$.
The second is that $E_x$ is a discrete $R[\stab_G(x)]$-module, so
every element is fixed by some  $N_i$.

For the second statement, consider a $G_\bullet$-sheaf $D_\bullet$ of $R$-modules over $X_\bullet$ 
whose adjoint structure maps
\[
\bar{\alpha}_{ij} \co D_j \lra (\bar{\phi}_{ij})_* D_i
\]
are isomorphisms. We want to show that 
\[
D_i \lra (\bar{p}_i)_* \colim_k (\bar{p}_k)^* D_k
\]
is an isomorphism of sheaves over the discrete space $X_i$. 
Choose $x \in X_i$, then we have isomorphisms as follows. 
Since we are looking at stalks of a sheaf over a discrete space, we do not 
need to consider the effect of any sheafification functors. 
The first step is the definition of the functors, 
the second is the fact that fixed points commute with filtered colimits 
as the terms of all the colimits are discrete modules. We are also able to 
replace $N_i$ by $N_i/N_k$ in the third step as $D_k$ is a $G/N_k$-sheaf, that is 
$G$ acts on $(\bar{p}_k)^* D_k$ via the projection to $G/N_k$.  
\begin{align*}
((\bar{p}_i)_* \colim_k (\bar{p}_k)^* D_k)_x 
&=
(\colim_k (\bar{p}_k)^* D_k)(p_i^{-1} \{x\})^{N_i}  \\
&\cong 
\colim_k D_k(p_k (p_i^{-1} \{x\}))^{N_i}  \\
&\cong 
\colim_k D_k(p_k (p_i^{-1} \{x\}))^{N_i/N_k}  \\
&=
\colim_k D_k(\phi_{ki}^{-1} \{x\})^{N_i/N_k}  \\
&=
\colim_k ((\bar{\phi}_{ki})_* D_k)(\{x\})  \\
&\cong
D_i(\{x\}) = (D_i)_x
\end{align*}
The penultimate isomorphism is where we use our assumption that the adjoint structure maps of $D_\bullet$
are isomorphisms.
\end{proof}

Fixing $X= \sub G$ and $R=\bQ$, we may ask if the adjunctions of Lemma \ref{lem:doublechangefunctor}
and Theorem \ref{thm:diagramsheafcolim} pass to categories of Weyl sheaves.

\begin{lemma}
Let $N$ be an open normal subgroup of $G$, with $p^G \co G \to G/N$ the projection
and $p \co \sub G \to \sub G/N$ the induced map on spaces of closed subgroups. 

The functor $\bar{p}^*$ sends Weyl-$G/N$-sheaves to Weyl-$G$-sheaves. 
Using $\bar{p}^*$ for the functor on Weyl-sheaves, there is an adjunction 
\[
\adjunction
{\bar{p}^*}
{\Rweylsheaf{G/N}{\bQ}}
{\Rweylsheaf{G}{\bQ}}
{\Weyl \circ \bar{p}_* \circ \inc}.
\]
\end{lemma}
\begin{proof}
Let $E$ be a Weyl-$G/N$-sheaf, and choose $K$ a closed subgroup of $G$. Then
\[
(\bar{p}^*E)_K
=
\varepsilon^* E_{pK}
=
\varepsilon^* E_{KN/N}.
\]
Since $E$ is a Weyl-$G/N$-sheaf, $E_{KN/N}$ is fixed by $KN/N = K/(K \cap N)$.
The group $K$ acts on $\varepsilon^* E_{KN/N}$ by passing to the quotient
and hence acts trivially. 

That we have an adjunction as stated follows from the preceding result and 
the fact that the inclusion of Weyl sheaves into equivariant sheaves is fully faithful, 
see Lemma \ref{lem:weyladjunction}.
\end{proof}

Combining this lemma with Theorem \ref{thm:diagramsheafcolim} gives the following corollary.

\begin{corollary}\label{cor:weyldiagrams}
Let $\sub_\bullet G$ be the diagram coming from the finite spaces $\sub G/N$ and the finite groups $G/N$, indexed over the 
diagram of open normal subgroups $N$ of $G$ and inclusions. 

There is an equivalence of categories
\[
\adjunction
{\underset{N \opennormalsub G}{\colim} \bar{p}_N^*}
{\colim \weylsheaf{G_\bullet}}
{\weylsheaf{G}}
{\bar{p}_\bullet}
\]
where the left hand category is the full subcategory of $\Reqsheaf{G_\bullet}{\sub G_\bullet}{\bQ}$
of those Weyl sheaves $D_\bullet$ of $\bQ$-modules over $\sub G_\bullet$ whose adjoint structure maps 
\[
\bar{\alpha}_{M,N} \co D_N \lra (\bar{\phi}_{M,N})_* D_M
\]
for open normal subgroups $M \leqslant N$ are isomorphisms. 
\end{corollary}

We can summarise the adjunctions of this paper into the following diagram, 
where arrows indicate inclusions of full subcategories.
\[
\xymatrix@C-0.35cm{
& 
\colim \weylsheaf{G_\bullet} \ar[r] \ar@{<->}[d]_-\cong \ar@{^{(}->}[r] 
&
\colim \Reqsheaf{G_\bullet}{\sub_\bullet G}{\bQ}
\ar@{<->}[d]_-\cong
\ar@{^{(}->}[r] 
&
\Reqsheaf{G_\bullet}{\sub_\bullet G}{\bQ}
\\
\mackey{G} \ar@{<->}[r]_-\cong &
\Rweylsheaf{G}{\bQ}
\ar@{^{(}->}[r] &
\Reqsheaf{G}{\sub G}{\bQ} 
\ar@{^{(}->}[r] &
\Reqpresheaf{G}{\bQ}{\sub G} 
}
\]


\begin{example}\label{ex:fixedpoints}
Let $V$ be a discrete $R[G]$-module, we want to make a Weyl-$G$-sheaf over $\sub G$
which at $K$ has stalk the $K$-fixed points of $V$. 
This is straightforward using the constructions of this section. 

For $N$ an open normal subgroup of $G$ there is a $G/N$-sheaf
$\textrm{Fix}_N(V)$ on the finite $G/N$-space $\sub G/N$ defined by 
\[
NK/N \mapsto V^{NK}.
\]
From these sheaves we define $\textrm{Fix}_{\bullet}(V)$, a $G_\bullet$-sheaf on $\sub G_\bullet$.
The structure maps
\[
\bar{\phi}_{M,N}^* \textrm{Fix}_N(V) \lra \textrm{Fix}_M(V)
\]
are induced by the inclusions $V^{NK} \longrightarrow V^{MK}$. 
Applying the colimit functor to this $G_\bullet$-sheaf 
gives a $G$-sheaf over $\sub G$
\[
\textrm{Fix}(V)=\colim_{N \opennormalsub G} \bar{p}_N^* \textrm{Fix}_{\bullet}(V).
\]
We calculate the stalks of $\textrm{Fix}(V)$ using the same
techniques as the proof of Theorem \ref{thm:diagramsheafcolim}.
The stalk of $\textrm{Fix}(V)$ at $K \in \sub G$ is 
\[
\colim_{N \opennormalsub G} V^{NK} \cong V^K.
\]
\end{example}

\appendix

\section{Profinite groups, profinite spaces and discrete modules}\label{Sec:profinitestuff}

We give a few reminders of useful facts on profinite groups, profinite spaces and discrete modules.
More details can be found in Wilson \cite{wilson98} or Ribes and Zalesskii \cite{rz00}.

A \emph{profinite group} is a compact, Hausdorff, totally disconnected
topological group.
A profinite group $G$ is homeomorphic to the inverse limit of its quotients by open normal subgroups:
\[
G \cong \underset{N \opennormalsub G}{\lim}  G/N \subseteq \underset{N \opennormalsub G}{\prod} G/N .
\]
The limit has the canonical topology which can either be described as the subspace topology on the product
or as the topology generated by the preimages of the open sets in
$G/N$ under the projection map $G \to G/N$, as $N$ runs over the open normal subgroups of $G$.
It follows that the open normal subgroups form a neighbourhood basis of the identity. 
Moreover, the intersection of all open normal subgroups is the trivial group, see also Lemma \ref{lem:intersections}.

Closed subgroups and quotients by closed normal subgroups of profinite groups are also profinite.
A subgroup of a profinite group is open if and only if it is of finite index and closed.
The trivial subgroup $\{ e\}$ is open if and only
if the group is finite.
Any open subgroup $H$ contains an open normal subgroup, the \emph{core} of $H$ in $G$,
which is defined as the finite intersection
\[
\core_G (H) = \bigcap_{g \in G} g H g^{-1}.
\]
If a (not necessarily closed) subgroup $H$ of $G$ contains a subgroup $K$ which is open in $G$, 
then $H$ itself is open. 

We can also define a \emph{profinite topological space}
to be a Hausdorff, compact and totally disconnected topological space.
As with profinite groups, such a space is homeomorphic to
the inverse limit of a cofiltered diagram of finite discrete spaces.
Moreover, a profinite topological space has an open-closed (hence compact-open) basis.

For $G$ a topological group, an action of $G$ on $X$ is 
a map $G \times X \to X$ which is associative and unital. 
We often refer to such an $X$ as a \emph{$G$-space}. 
A \emph{map of $G$-spaces} $f \co X \lra Y$ is a 
map of topological spaces which commutes with the action of $G$, that is, 
$f(gx) = g f(x)$ for all $x \in X$ and $g \in G$. 
We often refer to such a map as a \emph{$G$-equivariant} map.

We now give a few key results about the action of a profinite group on topological spaces.

\begin{lemma}\label{lem:restrictequi}
Let $G$ be a profinite group and $X$ a $G$-space. 
If $U$ is an open compact subset of $X$, then there exists an open
normal subgroup $N$ of $G$ such that $NU= U$. 
\end{lemma}
\begin{proof}
By continuity of the action map and the open normal subgroups forming a 
neighbourhood basis of the identity, for each $x \in U$ there 
are open normal subgroups $N_x$ and open neighbourhoods
$U_x \ni x$ such that $N_x U_x \subseteq U$. 
Take a finite subcover $\left\lbrace U_{x_i}\mid 1\leq i\leq n\right\rbrace$ of $U$,
then
\[
N=\underset{1\leq i\leq n}{\bigcap}N_{x_i}
\] 
is an open normal subgroup of $G$
which satisfies $N U =U$.
\end{proof}

We may think of this as an application of the fact that profinite groups have many ``small subgroups''. 
By way of comparison, Lie groups have no non-trivial ``small subgroups''. 
For example, the proper closed subgroups of $S^1$ are the finite cyclic groups.
Thus, a (sufficiently)  small neighbourhood of the identity of a Lie group 
will only contain the trivial subgroup.

\begin{lemma}\label{lem:openstab}
Let $G$ be a profinite group and $X$ a $G$-space.
If $U\subseteq X$ is compact and open, then the stabiliser subgroup of $U$
\begin{align*}
\stab_G(U)=\left\lbrace g\in G\mid gU= U\right\rbrace
\end{align*}
is an open subgroup of $G$. Hence, $\stab_G(U)$ has finite index in $G$.
\end{lemma}
\begin{proof}
The preceding lemma guarantees an open normal subgroup is contained in $\stab_G(U)$, 
hence $\stab_G(U)$ is open. 
\end{proof}

We now look at the action of profinite groups on $R$-modules, where $R$ is a 
unital ring. We assume that $R$ has trivial $G$-action. 
  
\begin{definition}\label{def:modulediscrete}
Let $M$ be a set with an (associative and unital) action of $G$.
We say that $M$ is \emph{discrete} if, when $M$ is equipped with the discrete topology, 
the action of $G$ on $M$ is continuous. 
A map of such sets is a $G$-equivariant map of sets.

Let $M$ be an $R$-module with an (associative and unital) action of $G$ that commutes with the action of $R$,
that is, an $R[G]$-module.
We say that $M$ is a \emph{discrete} $R[G]$-module if, when $M$ is equipped with the discrete topology, 
the action of $G$ on $M$ is continuous.  
The category of $R[G]$-modules is denoted $\Gmod{R}{G}$ and the full subcategory of 
discrete $R[G]$-modules is denoted $\Gmoddisc{R}{G}$. 
\end{definition}

We can describe discreteness in more algebraic terms and produce a discretisation functor
using this description. We give the following results for the case of $R[G]$-modules, 
but similar statements can be made for $G$-sets.

\begin{lemma}\label{lem:discretefixedpoints}
Let $G$ be a profinite group and $M$ an abelian group with an action of $G$. 
Then $M$ is discrete if and only if the canonical map
\[
\colim_{N \opennormalsub G} M^N \cong \colim_{H \opensub G} M^H \overset{\cong}{\lra} M
\]
is an isomorphism.
\end{lemma}
The isomorphism of colimits is a cofinality argument, 
based on the fact that every open subgroup contains an open normal subgroup. 
The lemma can be rephrased as saying that $M$ is discrete if and only if 
for each element  $m \in M$ there is an open (normal)
subgroup $H$ of $G$ such that $m$ is $H$-fixed. 
Equally, we may say the the stabiliser of any point is open. 

\begin{examples}
The primary examples of discrete $R[G]$-modules are the modules $R[G/H]$ 
(the $R$-module with basis the cosets of $H$ in $G$) for $H$ an open subgroup of $G$.
To see this is discrete, note that the basis element $gH$ is fixed by the open subgroup $gHg^{-1}$.
A finite sum of such elements is fixed by the intersection of those subgroups, which is also open. 
Hence, every element is fixed by some open subgroup.  
Conversely, the module $R[G]$ is not discrete when $G$ is infinite as the basis element $e$ is 
not fixed by any open subgroup.  
\end{examples}

Given any $R[G]$-module $M$, the colimit of Lemma \ref{lem:discretefixedpoints} defines a discrete $R[G]$-module, 
as the following lemma states. 
This construction is part of an adjunction which also allows for a description of
limits in the category of discrete $R[G]$-modules. 
\begin{lemma}
For $G$ a profinite group, the inclusion 
\[
\Gmoddisc{R}{G} \lra \Gmod{R}{G}
\]
has a right adjoint, $\disc$, which is defined by 
\[
\disc (M) = \colim_{H \opensub G} M^H.
\]
We call this functor \emph{discretisation}.
The map 
\[
\disc (M) \lra M
\]
is a monomorphism, which is an isomorphism if $M$ is already discrete.
We sometimes call $\disc (M)$ the \emph{discrete part} of $M$.

The category of $R[G]$-modules has all small limits and colimits, 
defined in terms of underlying $R$-modules with induced $G$-actions.
Colimits in $\Gmoddisc{R}{G}$ are constructed similarly. 
A limit in $\Gmoddisc{R}{G}$ is given by constructing the limit in $\Gmod{R}{G}$
and then applying $\disc$. 
\end{lemma}

That the limit is as described follows from the fact that the image of a 
discrete module under a $G$-equivariant map is contained in the discrete part of the codomain. 

\begin{remark}
Note that nontrivial modules can have trivial discretisation, such as $\bQ[G]$ when $G$ is infinite. 
Indeed, an element of $\bQ[G]$ is a finite sum of (rational multiples of) basis elements.
Since the sum is finite, it cannot be fixed by an open subgroup (which is necessarily infinite). 
\end{remark}

Since finite limits commute with filtered colimits in  module categories, 
a finite limit
in $\Gmoddisc{R}{G}$ is given by the limit in $\Gmod{R}{G}$. The problem lies with the infinite product. 
Indeed, let $G = \adic{p}$ (for $p$ a prime) and consider the following product 
of discrete $\adic{p}$-modules.
\[
\prod_{n \geqslant 0} R[\bZ/p^n].
\]
Recall that the open subgroups of $\adic{p}$ are the subgroups $p^k \adic{p}$, for $k \geqslant 0$.
The element of the product consisting of 1 in each factor is not fixed by any open subgroup. 
Of course, we already expected a problem like this, as the infinite product of 
discrete spaces is not discrete. Similarly, an infinite product of discrete $R[G]$-modules would have to be 
more complicated than the underlying product.  

We finish this section by noting that we have enough injective discrete $R[G]$-modules. 
\begin{lemma}\label{lem:rgmodinjectives}
For  $G$ a profinite group and $R$ a ring, the category of discrete $R[G]$-modules has enough injectives.
\end{lemma}
\begin{proof}
The inclusion functor from 
discrete $R[G]$-modules to $R[G]$-modules preserves monomorphisms.
Hence, its right adjoint (discretisation) preserves injective objects. 
Let $A$ be a discrete $R[G]$-module, then as 
the category of $R[G]$-modules has enough injectives,
there is an injective $R[G]$-module $I$ and a monomorphism $i \co A \lra I$. 
Since the image of $A$ in $I$ is discrete, $i$ factors over a monomorphism
$A \lra \disc (I)$. 
\end{proof}

\section{Sheaves and presheaves}\label{Sec:sheavespresheaves}

We give the basic definitions and terminology, Swan \cite{swan}, Tennison \cite{tenn} 
or Bredon \cite{bredonsheaf} provide more details.

\begin{definition}\label{def:sheafandstalk}
Let $X$ be a topological space. 
A \emph{presheaf of sets over $X$} is a contravariant functor 
from the category of open sets of $X$ and inclusions to sets
\[
F \co \mcO(X)^{op} \lra \sets.
\]
A \emph{map of presheaves} is a natural transformation. 

We call an element of $F(U)$ a \emph{section} of $F$ supported on $U$. 
The \emph{stalk} of $F$ at $x \in X$ is the following colimit
\[
F_x = \colim_{U \ni x} F(U)
\]
and an element of the stalk is called a \emph{germ}.
For $s \in F(U)$, the image of $s$ in $F_x$ is denoted $s_x$.
\end{definition}

\begin{definition}
A \emph{sheaf of sets over $X$} is a presheaf of sets over $X$ 
that satisfies the \emph{sheaf condition}:
for each cover of an open set $U$, 
$U = \cup_{\lambda \in \Lambda} U_\lambda$, 
we have an equaliser diagram
\[
\xymatrix{
F(U) \cong \textnormal{eq} \Bigg(  
\displaystyle\coprod_{\lambda \in \Lambda} F(U_\lambda) 
\ar@<0.7ex>[r] \ar@<-0.7ex>[r]   &  
\displaystyle\coprod_{(\lambda,\mu) \in \Lambda \times \Lambda} F(U_\lambda \cap U_\mu)
\Bigg)  
}
\]
where the maps are induced by the restriction maps.
A \emph{map of sheaves} over $X$ is a map of underlying presheaves.
\end{definition}

\begin{remark}
If we write the equaliser as $\lim_{\lambda \in \Lambda} F(U_\lambda)$, then we may describe the 
the sheaf condition as an isomorphism 
\[
F(U) \lra \lim_{\lambda \in \Lambda} F(U_\lambda).
\]
That this map is surjective is also known as the \emph{patching condition},
which says that a set of sections $s_\lambda \in F(U_\lambda)$,
which agree under restrictions to intersections, can be patched into a 
section $s \in F(U)$ which restricts to each $s_\lambda$. 
That the above map is injective is also known as the \emph{separation condition}, which says that 
if two sections $s, t \in F(U)$ agree on each $F(U_\lambda)$ then they are the same section. 
\end{remark}

\begin{definition}
A \emph{local homeomorphism} is a continuous map $p \co E \to X$ such that 
for every $e \in E$, there is an open set $V \ni e$ of $E$ and an open set 
$U \ni p(e)$ of $X$ such that $p_{|V}$ is 
a homeomorphism onto $U$. 
We call $X$ the \emph{base space} of the local homeomorphism.
For $X$ a topological space, a \emph{sheaf space over $X$} is a 
local homeomorphism $p \co E \to X$.

Given a map of topological spaces $p \co E \to X$, 
a \emph{section} of $p$ over an open set $U \subseteq X$ is a continuous map $s \co U \to E$
such that $p \circ s=\id_U$. By convention, we write $s_x$ for the value of $s$ at $x \in U$. 
\end{definition}

In particular, local homeomorphisms are open maps
and the images of sections over the open sets of $X$ form a basis for the topology
of $E$, see Tennison \cite[Lemma 2.3.5]{tenn}.

\begin{lemma}
If $p \co E \to X$ is a local homeomorphism, then the preimage of a point $x \in X$ 
(as a subspace of $E$) is a discrete topological space.
\end{lemma}

\begin{lemma}
A sheaf of sets over $X$ defines a sheaf space over $X$.
Conversely, a sheaf space over $X$ defines a sheaf of sets over $X$. 
These two constructions are mutually inverse. 
\end{lemma}


As the constructions are inverse to each other, we see that a section of 
a sheaf space is exactly a section of the corresponding sheaf. 
That is, given a sheaf $F$, we can make a sheaf space $E$, 
and see that sections $F(U)$ are precisely the sections of the sheaf space over $U$: $\Gamma( U, E)$. 
By Tennison \cite[Proposition 2.3.6]{tenn}, the stalk of a sheaf is the preimage of the corresponding 
local homeomorphism:
\[
\underset{U \ni x}{\colim \,} \Gamma(U,E) = E_x=p^{-1} (x) = \colim_{U \ni x} F(U) = F_x.
\]

The construction of a sheaf space did not require that the input was a sheaf, only
that we had a presheaf. thus we may take a presheaf $F$, construct a sheaf space from it $E$
and then construct a sheaf from $E$. We formalise this in the following. 

\begin{definition}
The \emph{sheafification} of a presheaf $F$ is the sheaf corresponding to the 
sheaf space of the presheaf, which we write as  $\sheafify F$. 
\end{definition}

\begin{lemma}\label{lem:noneqsheafify}
The sheafification construction $\sheafify$ is a functor that is left adjoint 
to the forgetful functor from sheaves to presheaves
(this functor is often omitted from the notation). 

The unit of  this adjunction gives a 
canonical map from a presheaf to a sheaf $F \to \sheafify F$. 
The canonical map $\sheafify F \to \sheafify^2 F$ is an isomorphism. 
Furthermore, if $E$ is a sheaf, a map $F \to E$ factors over $F \to \sheafify F$.
\end{lemma}

We can alter the definition of a presheaf 
to land in abelian groups or modules over a ring $R$. 
We then define sheaves and sheaf spaces of $R$-modules
in a similar fashion. 

\begin{definition}
Let $X$ be a topological space. 
A \emph{presheaf of $R$-modules over $X$} is a contravariant functor 
from the category of open sets of $X$ and inclusions to sets
\[
F \co \mcO(X)^{op} \lra R \leftmod.
\]
A \emph{sheaf of $R$-modules over $X$} is a sheaf whose underlying presheaf is a 
presheaf of $R$-modules over $X$. 

A \emph{sheaf space of $R$-modules over $X$} 
is a sheaf space whose underlying presheaf is a presheaf of $R$-modules over $X$. 
\end{definition}

The given definition of a sheaf space of $R$-modules is not as useful as one would like,
so a more intrinsic characterisation is given below. 

\begin{lemma}
A sheaf space $p \co E \to X$ is a sheaf space of $R$-modules if and only if
for each $x \in X$, $p^{-1} (x)$ is an $R$-module and the map 
\begin{align*}
E \underset{X}{\times} E  = \{ (e,e') \mid p(e)=p(e') \} & \lra E \\
(e,e') & \longmapsto e-e'
\end{align*}
induced by the pointwise group operation is continuous and for $r \in R$, the map 
$e \mapsto re$ is a continuous self map of $E$. 
\end{lemma}

The results on sheaves, sheaf spaces and sheafification all extend to 
$R$-module variants. From here on we will no longer be as precise about 
the difference between a sheaf and a sheaf space. 

We will find it useful to restrict our definition of presheaves to a basis
in Section \ref{sec:equivpresheaves} where we define equivariant presheaves. 
The definition relies on the following result of Tennison \cite[Lemma 4.2.6]{tenn}.
It says that a sheaf over a topological space $X$ is uniquely determined by its 
behaviour on a basis for the topology of $X$.  
The lemma is stated for (pre)sheaves of sets but readily extends to sheaves of $R$-modules.

\begin{lemma}\label{lem:sheafoverbasis}
Let $X$ be a topological space with $\mcB$ a basis for the topology.
Write $\mcO(\mcB)$ to be the category of elements of $\mcB$
and inclusions. 

Assume there is a contravariant functor 
\[
F \co \mcO(\mcB)^{op} \lra \sets
\]
such that whenever $U = \cup_{\lambda \in \Lambda} U_\lambda$ in $\mcO(\mcB)$, 
we have an equaliser diagram
\[
\xymatrix{
F(U) = \textnormal{eq} \Bigg(  
\displaystyle\coprod_{\lambda \in \Lambda} F(U_\lambda) 
\ar@<0.7ex>[r] \ar@<-0.7ex>[r]   &  
\displaystyle\coprod_{(\lambda,\mu) \in \Lambda \times \Lambda} F(U_\lambda \cap U_\mu)
\Bigg).  
}
\]
Then there is a sheaf $F'$ on $X$, such that 
$F'$ restricted to $\mcO(\mcB)^{op}$ is equal to $F$
Moreover, this sheaf $F'$ is unique up to isomorphism. 
\end{lemma}

The proof is based on the idea that one can define a sheaf space from $F$ as for an ordinary (pre)sheaf. 
One also sees that 
\[
F'(V) = \lim_{\mcO(\mcB) \ni U \subseteq V} F(U)
\]
describes $F'$ as a right Kan extension over the inclusion $\mcO(\mcB) \to \mcO(X)$.

We will need to consider change of base space constructions on sheaves, 
particularly in the case of the maps $g \co X \lra X$ coming from the action of a group $G$ on $X$. 
Again we state the theorem for sets, but it can be written for $R$-modules just as well. 

\begin{definition}\label{defn:pullpushfunctors}
Let $f \co X \to Y$ be a map of topological spaces.
Let $F$ be a presheaf over $X$ and let $E$ be a sheaf space over $Y$.

We define the presheaf $f_*(F)$ over $Y$, the \emph{push forward of $F$ over $f$}, to be the functor
\[
\mcO(Y)^{op} \overset{f^{-1}}{\lra} \mcO(X)^{op} \overset{F}{\lra} \sets
\]
where $f^{-1}$ is the pre-image functor induced by $f$. 

We define a sheaf space $f^*(E)$ over $X$, the \emph{pull back of $E$ over $f$} by the pullback
\[
\xymatrix{
f^*(E) \ar[r] \ar[d] &
E \ar[d]_p \\
X \ar[r]_f & Y.
}
\]
\end{definition}

\begin{proposition}\label{prop:pushpulladjunct}
Let $f \co X \to Y$ be a map of topological spaces.
If $F$ is a sheaf over $X$, then $f_*(F)$ is a sheaf over $Y$. 
If $E$ is a sheaf space over $Y$, then $f^*(E)$ is a sheaf space over $X$. 
Given $x \in X$, there is an isomorphism of stalks
\[
f^*(E)_x \cong E_{f(x)}.
\]

The functors $f^*$ and $f_*$ form an adjunction
\[
\adjunction{f^*}{\setsheaf{Y}}{\setsheaf{X}}{f_*}
\]
with the left adjoint $f^*$ preserving finite limits. 
\end{proposition}

That $f^*$ preserves finite limits follows from the analogous statement for 
categories of over-spaces, see also Construction \ref{limconstruct}.

\begin{definition}\label{defn:extres}
When $f \co X \to Y$ is an inclusion of a subspace, we call $f^*$ the \emph{restriction to $X$ functor}.

When working with sheaves of $R$-modules, if 
$f \co X \to Y$ is an inclusion of a closed subspace, 
then $f_*$ is called \emph{extension by zero}. 
\end{definition}

Given a sheaf $F$ over $X$ and an inclusion of a closed subspace 
$f \co X \to Y$, the stalks of the extension by zero $f_* F$ are zero outside $X$. 
Moreover, the restriction of the extension by zero 
$f^* f_* F$, is isomorphic to $F$. 
We also note that if $f$ is an open map and $F$ a sheaf over $Y$, then 
the pull back sheaf $f^*F$ can be defined by $f^*F(U) =  F(f(U))$.

\bibliographystyle{alpha}
\bibliography{ourbib}

\begin{thebibliography}{LMSM86}

\bibitem[BB04]{BB04}
W.~Bley and R.~Boltje.
\newblock Cohomological {M}ackey functors in number theory.
\newblock {\em J. Number Theory}, 105(1):1--37, 2004.

\bibitem[BK22]{BKclassify}
D.~Barnes and M.~K\k{e}dziorek.
\newblock An introduction to algebraic models for rational {$G$}-spectra.
\newblock In {\em Equivariant topology and derived algebra}, volume 474 of {\em
  London Math. Soc. Lecture Note Ser.}, pages 119--179. Cambridge Univ. Press,
  Cambridge, 2022.

\bibitem[BL94]{bernlunts}
J.~Bernstein and V.~Lunts.
\newblock {\em Equivariant sheaves and functors}, volume 1578 of {\em Lecture
  Notes in Mathematics}.
\newblock Springer-Verlag, Berlin, 1994.

\bibitem[Bre97]{bredonsheaf}
G.~E. Bredon.
\newblock {\em Sheaf theory}, volume 170 of {\em Graduate Texts in
  Mathematics}.
\newblock Springer-Verlag, New York, second edition, 1997.

\bibitem[BS20]{BSmackey}
D.~Barnes and D.~Sugrue.
\newblock The equivalence between rational {$G$}-sheaves and rational
  {$G$}-mackey functors for profinite {$G$}.
\newblock {\tt arXiv:2002.11745}, 2020.

\bibitem[BS22]{BSspectra}
D.~Barnes and D.~Sugrue.
\newblock Classifying rational {$G$}-spectra for profinite {$G$}.
\newblock In preparation, 2022.

\bibitem[GL01]{GL01}
T.~G. Goodwillie and S.~Lichtenbaum.
\newblock A cohomological bound for the {$h$}-topology.
\newblock {\em Amer. J. Math.}, 123(3):425--443, 2001.

\bibitem[Gre]{Gconjecture}
J.~P.~C. Greenlees.
\newblock Triangulated categories of rational equivariant cohomology theories.
\newblock {Oberwolfach Reports 8/2006, 480-488}.

\bibitem[GS10a]{GartSmith10}
P.~Gartside and M.~Smith.
\newblock Classifying spaces of subgroups of profinite groups.
\newblock {\em J. Group Theory}, 13(3):315--336, 2010.

\bibitem[GS10b]{GSclassify}
P.~Gartside and M.~Smith.
\newblock Classifying spaces of subgroups of profinite groups.
\newblock {\em J. Group Theory}, 13(3):315--336, 2010.

\bibitem[Joh02]{elephant}
P.~T. Johnstone.
\newblock {\em Sketches of an elephant: a topos theory compendium. {V}ol. 1},
  volume~43 of {\em Oxford Logic Guides}.
\newblock The Clarendon Press, Oxford University Press, New York, 2002.

\bibitem[K{\k{e}}d14]{kedziorekthesis}
M.~K{\k{e}}dziorek.
\newblock {\em Algebraic models for rational {$G$}-spectra}.
\newblock PhD thesis, University of Sheffield, 2014.

\bibitem[LMSM86]{lms86}
L.~G. Lewis, Jr., J.~P. May, M.~Steinberger, and J.~E. McClure.
\newblock {\em Equivariant stable homotopy theory}, volume 1213 of {\em Lecture
  Notes in Mathematics}.
\newblock Springer-Verlag, Berlin, 1986.
\newblock With contributions by J. E. McClure.

\bibitem[Moe88]{moer88}
I.~Moerdijk.
\newblock The classifying topos of a continuous groupoid. {I}.
\newblock {\em Trans. Amer. Math. Soc.}, 310(2):629--668, 1988.

\bibitem[Moe90]{moer90}
I.~Moerdijk.
\newblock The classifying topos of a continuous groupoid. {II}.
\newblock {\em Cahiers Topologie G\'{e}om. Diff\'{e}rentielle Cat\'{e}g.},
  31(2):137--168, 1990.

\bibitem[RZ00]{rz00}
L.~Ribes and P.~Zalesskii.
\newblock {\em Profinite groups}, volume~40 of {\em Ergebnisse der Mathematik
  und ihrer Grenzgebiete. 3. Folge. A Series of Modern Surveys in Mathematics
  [Results in Mathematics and Related Areas. 3rd Series. A Series of Modern
  Surveys in Mathematics]}.
\newblock Springer-Verlag, Berlin, 2000.

\bibitem[Sch94]{Scheiderer94}
C.~Scheiderer.
\newblock {\em Real and \'{e}tale cohomology}, volume 1588 of {\em Lecture
  Notes in Mathematics}.
\newblock Springer-Verlag, Berlin, 1994.

\bibitem[Sch98]{schneider98}
P.~Schneider.
\newblock Equivariant homology for totally disconnected groups.
\newblock {\em J. Algebra}, 203(1):50--68, 1998.

\bibitem[Sug19]{sugruethesis}
D.~Sugrue.
\newblock Rational {$G$}-spectra for profinite {$G$}.
\newblock {\tt arXiv} 1910.12951, 2019.

\bibitem[Swa64]{swan}
R.~G. Swan.
\newblock {\em The Theory of Sheaves}.
\newblock The University of Chicago Press, 1964.
\newblock Chicago Lectures In Mathematics.

\bibitem[Ten75]{tenn}
B.~R. Tennison.
\newblock {\em Sheaf theory}.
\newblock Cambridge University Press, Cambridge, England, 1975.
\newblock London Mathematical Society Lecture Note Series, No. 20.

\bibitem[Wil98]{wilson98}
J.~S. Wilson.
\newblock {\em Profinite groups}, volume~19 of {\em London Mathematical Society
  Monographs. New Series}.
\newblock The Clarendon Press, Oxford University Press, New York, 1998.

\end{thebibliography}

\end{document}